\documentclass[a4paper,11pt]{article}
\usepackage{MyStyle}

\usepackage[hidelinks]{hyperref}
\usepackage{leftindex}
\usepackage{mathtools}
\usepackage{amsmath}
\usepackage{shuffle}
\usepackage[normalem]{ulem}
\usepackage{nameref} 

\usepackage[all]{xy}

\newcolumntype{C}{>{$}c<{$}}
\usepackage{planarforest}

\newcommand{\forestA}{
\tikz[planar forest] {

\node [b] at (0.0, 0.0) {  } 
;
}}

\newcommand{\forestB}{
\tikz[planar forest] {

\node [b] at (0.0, 0.0) {  } 
;
}}

\newcommand{\forestC}{
\tikz[planar forest] {

\node [b] at (0.0, 0.0) {  } 
;
}}

\newcommand{\forestD}{
\tikz[planar forest] {

\node [b] at (0.0, 0.0) {  } 
;
}}

\newcommand{\forestE}{
\tikz[planar forest] {

\node [b] at (0.0, 0.0) {  } 
child {node [b] at (-0.5, 1.0) {  }  
}
child {node [b] at (0.5, 1.0) {  }  
child {node [b] at (0.0, 1.0) {  }  
}
}
;
}}

\newcommand{\forestF}{
\tikz[planar forest] {

\node [b] at (0.0, 0.0) {  } 
child {node [b] at (-0.5, 1.0) {  }  
child {node [b] at (0.0, 1.0) {  }  
}
}
child {node [b] at (0.5, 1.0) {  }  
}
;
}}

\newcommand{\forestG}{
\tikz[planar forest] {

\node [b] at (0.0, 0.0) {  } 
child {node [b] at (0.0, 1.0) {  }  
}
;
}}

\newcommand{\forestH}{
\tikz[planar forest] {

\node [b] at (0.0, 0.0) {  } 
child {node [b] at (0.0, 1.0) {  }  
}
;
}}

\newcommand{\forestI}{
\tikz[planar forest] {

\node [b] at (0.0, 0.0) {  } 
child {node [b] at (-0.5, 1.0) {  }  
}
child {node [b] at (0.5, 1.0) {  }  
child {node [b] at (0.0, 1.0) {  }  
}
}
;
}}

\newcommand{\forestJ}{
\tikz[planar forest] {

\node [b] at (0.0, 0.0) {  } 
child {node [b] at (0.0, 1.0) {  }  
child {node [b] at (0.0, 1.0) {  }  
child {node [b] at (0.0, 1.0) {  }  
}
}
}
;
}}

\newcommand{\forestK}{
\tikz[planar forest] {

\node [b] at (0.0, 0.0) {  } 
;
}}

\newcommand{\forestL}{
\tikz[planar forest] {

\node [b] at (0.0, 0.0) {  } 
child {node [b] at (-0.5, 1.0) {  }  
}
child {node [b] at (0.5, 1.0) {  }  
}
;
}}

\newcommand{\forestM}{
\tikz[planar forest] {

\node [b] at (0.0, 0.0) {  } 
child {node [b] at (-1.0, 1.0) {  }  
}
child {node [b] at (0.0, 1.0) {  }  
}
child {node [b] at (1.0, 1.0) {  }  
}
;
}}

\newcommand{\forestN}{
\tikz[planar forest] {

\node [b] at (0.0, 0.0) {  } 
child {node [b] at (-0.5, 1.0) {  }  
}
child {node [b] at (0.5, 1.0) {  }  
child {node [b] at (0.0, 1.0) {  }  
}
}
;
}}

\newcommand{\forestO}{
\tikz[planar forest] {

\node [b] at (0.0, 0.0) {  } 
child {node [b] at (-0.5, 1.0) {  }  
}
child {node [b] at (0.5, 1.0) {  }  
}
;
}}

\newcommand{\forestP}{
\tikz[planar forest] {

\node [b] at (0.0, 0.0) {  } 
;
}}

\newcommand{\forestQ}{
\tikz[planar forest] {

\node [b] at (0.0, 0.0) {  } 
child {node [b] at (0.0, 1.0) {  }  
child {node [b] at (-0.5, 1.0) {  }  
}
child {node [b] at (0.5, 1.0) {  }  
}
}
;
}}

\newcommand{\forestR}{
\tikz[planar forest] {

\node [b] at (0.0, 0.0) {  } 
child {node [b] at (-0.5, 1.0) {  }  
}
child {node [b] at (0.5, 1.0) {  }  
}
;
}}

\newcommand{\forestS}{
\tikz[planar forest] {

\node [b] at (0.0, 0.0) {  } 
;
}}

\newcommand{\forestT}{
\tikz[planar forest] {

\node [b] at (0.0, 0.0) {  } 
;
}}

\newcommand{\forestU}{
\tikz[planar forest] {

\node [b] at (0.0, 0.0) {  } 
;
}}

\newcommand{\forestV}{
\tikz[planar forest] {

\node [b] at (0.0, 0.0) {  } 
;
}}

\newcommand{\forestW}{
\tikz[planar forest] {

\node [b] at (0.0, 0.0) {  } 
;
}}

\newcommand{\forestX}{
\tikz[planar forest] {

\node [b] at (0.0, 0.0) {  } 
;
}}

\newcommand{\forestY}{
\tikz[planar forest] {

\node [b] at (0.0, 0.0) {  } 
;
}}

\newcommand{\forestAB}{
\tikz[planar forest] {

\node [b] at (0.0, 0.0) {  } 
;
}}

\newcommand{\forestBB}{
\tikz[planar forest] {

\node [b] at (0.0, 0.0) {  } 
;
}}

\newcommand{\forestCB}{
\tikz[planar forest] {

\node [b] at (0.0, 0.0) {  } 
;
}}

\newcommand{\forestDB}{
\tikz[planar forest] {

\node [b] at (0.0, 0.0) {  } 
;
}}

\newcommand{\forestEB}{
\tikz[planar forest] {

\node [b] at (0.0, 0.0) {  } 
;
}}

\newcommand{\forestFB}{
\tikz[planar forest] {

\node [b] at (0.0, 0.0) {  } 
;
}}

\newcommand{\forestGB}{
\tikz[planar forest] {

\node [b] at (0.0, 0.0) {  } 
;
}}

\newcommand{\forestHB}{
\tikz[planar forest] {

\node [b] at (0.0, 0.0) {  } 
;
}}

\newcommand{\forestIB}{
\tikz[planar forest] {

\node [b] at (0.0, 0.0) {  } 
;
}}

\newcommand{\forestJB}{
\tikz[planar forest] {

\node [b] at (0.0, 0.0) {  } 
;
}}

\newcommand{\forestKB}{
\tikz[planar forest] {

\node [b] at (0.0, 0.0) {  } 
child {node [b] at (0.0, 1.0) {  }  
}
;
}}

\newcommand{\forestLB}{
\tikz[planar forest] {

\node [b] at (0.0, 0.0) {  } 
child {node [b] at (0.0, 1.0) {  }  
child {node [b] at (0.0, 1.0) {  }  
}
}
;
}}

\newcommand{\forestMB}{
\tikz[planar forest] {

\node [b] at (0.0, 0.0) {  } 
child {node [b] at (-0.5, 1.0) {  }  
}
child {node [b] at (0.5, 1.0) {  }  
}
;
}}

\newcommand{\forestNB}{
\tikz[planar forest] {

\node [b] at (0.0, 0.0) {  } 
child {node [b] at (0.0, 1.0) {  }  
child {node [b] at (0.0, 1.0) {  }  
child {node [b] at (0.0, 1.0) {  }  
}
}
}
;
}}

\newcommand{\forestOB}{
\tikz[planar forest] {

\node [b] at (0.0, 0.0) {  } 
child {node [b] at (0.0, 1.0) {  }  
child {node [b] at (-0.5, 1.0) {  }  
}
child {node [b] at (0.5, 1.0) {  }  
}
}
;
}}

\newcommand{\forestPB}{
\tikz[planar forest] {

\node [b] at (0.0, 0.0) {  } 
child {node [b] at (-0.5, 1.0) {  }  
}
child {node [b] at (0.5, 1.0) {  }  
child {node [b] at (0.0, 1.0) {  }  
}
}
;
}}

\newcommand{\forestQB}{
\tikz[planar forest] {

\node [b] at (0.0, 0.0) {  } 
child {node [b] at (-1.0, 1.0) {  }  
}
child {node [b] at (0.0, 1.0) {  }  
}
child {node [b] at (1.0, 1.0) {  }  
}
;
}}

\newcommand{\forestRB}{
\tikz[planar forest] {

\node [b] at (0.0, 0.0) {  } 
child {node [b] at (0.0, 1.0) {  }  
}
;
}}

\newcommand{\forestSB}{
\tikz[planar forest] {

\node [b] at (0.0, 0.0) {  } 
;
}}

\newcommand{\forestTB}{
\tikz[planar forest] {

\node [b] at (0.0, 0.0) {  } 
child {node [b] at (-0.5, 1.0) {  }  
}
child {node [b] at (0.5, 1.0) {  }  
}
;
}}

\newcommand{\forestUB}{
\tikz[planar forest] {

\node [b] at (0.0, 0.0) {  } 
child {node [b] at (0.0, 1.0) {  }  
child {node [b] at (-0.5, 1.0) {  }  
}
child {node [b] at (0.5, 1.0) {  }  
}
}
;
}}

\newcommand{\forestVB}{
\tikz[planar forest] {

\node [b] at (0.0, 0.0) {  } 
child {node [b] at (-1.0, 1.0) {  }  
}
child {node [b] at (0.0, 1.0) {  }  
}
child {node [b] at (1.0, 1.0) {  }  
}
;
}}

\newcommand{\forestWB}{
\tikz[planar forest] {

\node [b] at (0.0, 0.0) {  } 
child {node [b] at (0.0, 1.0) {  }  
child {node [b] at (0.0, 1.0) {  }  
child {node [b] at (-0.5, 1.0) {  }  
}
child {node [b] at (0.5, 1.0) {  }  
}
}
}
;
}}

\newcommand{\forestXB}{
\tikz[planar forest] {

\node [b] at (0.0, 0.0) {  } 
child {node [b] at (0.0, 1.0) {  }  
child {node [b] at (-1.0, 1.0) {  }  
}
child {node [b] at (0.0, 1.0) {  }  
}
child {node [b] at (1.0, 1.0) {  }  
}
}
;
}}

\newcommand{\forestYB}{
\tikz[planar forest] {

\node [b] at (0.0, 0.0) {  } 
child {node [b] at (-0.5, 1.0) {  }  
}
child {node [b] at (0.5, 1.0) {  }  
child {node [b] at (-0.5, 1.0) {  }  
}
child {node [b] at (0.5, 1.0) {  }  
}
}
;
}}

\newcommand{\forestAC}{
\tikz[planar forest] {

\node [b] at (0.0, 0.0) {  } 
child {node [b] at (-1.5, 1.0) {  }  
}
child {node [b] at (-0.5, 1.0) {  }  
}
child {node [b] at (0.5, 1.0) {  }  
}
child {node [b] at (1.5, 1.0) {  }  
}
;
}}

\newcommand{\forestBC}{
\tikz[planar forest] {

\node [b] at (0.0, 0.0) {  } 
child {node [b] at (0.0, 1.0) {  }  
child {node [b] at (0.0, 1.0) {  }  
}
}
;
}}

\newcommand{\forestCC}{
\tikz[planar forest] {

\node [b] at (0.0, 0.0) {  } 
child {node [b] at (0.0, 1.0) {  }  
child {node [b] at (0.0, 1.0) {  }  
}
}
;
}}

\newcommand{\forestDC}{
\tikz[planar forest] {

\node [b] at (0.0, 0.0) {  } 
;
}}

\newcommand{\forestEC}{
\tikz[planar forest] {

\node [b] at (0.0, 0.0) {  } 
child {node [b] at (0.0, 1.0) {  }  
}
;
}}

\newcommand{\forestFC}{
\tikz[planar forest] {

\node [b] at (0.0, 0.0) {  } 
child {node [b] at (0.0, 1.0) {  }  
}
;
}}

\newcommand{\forestGC}{
\tikz[planar forest] {

\node [b] at (0.0, 0.0) {  } 
;
}}

\newcommand{\forestHC}{
\tikz[planar forest] {

\node [b] at (0.0, 0.0) {  } 
child {node [b] at (0.0, 1.0) {  }  
child {node [b] at (0.0, 1.0) {  }  
}
}
;
}}

\newcommand{\forestIC}{
\tikz[planar forest] {

\node [b] at (0.0, 0.0) {  } 
child {node [b] at (-0.5, 1.0) {  }  
}
child {node [b] at (0.5, 1.0) {  }  
child {node [b] at (0.0, 1.0) {  }  
}
}
;
}}

\newcommand{\forestJC}{
\tikz[planar forest] {

\node [b] at (0.0, 0.0) {  } 
child {node [b] at (-0.5, 1.0) {  }  
}
child {node [b] at (0.5, 1.0) {  }  
child {node [b] at (0.0, 1.0) {  }  
}
}
;
}}

\newcommand{\forestKC}{
\tikz[planar forest] {

\node [b] at (0.0, 0.0) {  } 
;
}}

\newcommand{\forestLC}{
\tikz[planar forest] {

\node [b] at (0.0, 0.0) {  } 
child {node [b] at (-0.5, 1.0) {  }  
}
child {node [b] at (0.5, 1.0) {  }  
}
;
}}

\newcommand{\forestMC}{
\tikz[planar forest] {

\node [b] at (0.0, 0.0) {  } 
;
}}

\newcommand{\forestNC}{
\tikz[planar forest] {

\node [b] at (0.0, 0.0) {  } 
child {node [b] at (0.0, 1.0) {  }  
child {node [b] at (0.0, 1.0) {  }  
}
}
;
}}

\newcommand{\forestOC}{
\tikz[planar forest] {

\node [b] at (0.0, 0.0) {  } 
;
\node [b] at (1.0, 0.0) {  } 
;
}}

\newcommand{\forestPC}{
\tikz[planar forest] {

\node [b] at (0.0, 0.0) {  } 
child {node [b] at (0.0, 1.0) {  }  
}
;
}}

\newcommand{\forestQC}{
\tikz[planar forest] {

\node [b] at (0.0, 0.0) {  } 
child {node [b] at (0.0, 1.0) {  }  
}
;
}}

\newcommand{\forestRC}{
\tikz[planar forest] {

\node [b] at (0.0, 0.0) {  } 
child {node [b] at (0.0, 1.0) {  }  
}
;
}}

\newcommand{\forestSC}{
\tikz[planar forest] {

\node [b] at (0.0, 0.0) {  } 
;
\node [b] at (1.0, 0.0) {  } 
child {node [b] at (0.0, 1.0) {  }  
}
;
}}

\newcommand{\forestTC}{
\tikz[planar forest] {

\node [b] at (0.0, 0.0) {  } 
;
}}

\newcommand{\forestUC}{
\tikz[planar forest] {

\node [b] at (0.0, 0.0) {  } 
child {node [b] at (-0.5, 1.0) {  }  
}
child {node [b] at (0.5, 1.0) {  }  
child {node [b] at (0.0, 1.0) {  }  
}
}
;
}}

\newcommand{\forestVC}{
\tikz[planar forest] {

\node [b] at (0.0, 0.0) {  } 
;
}}

\newcommand{\forestWC}{
\tikz[planar forest] {

\node [b] at (0.0, 0.0) {  } 
child {node [b] at (-0.5, 1.0) {  }  
}
child {node [b] at (0.5, 1.0) {  }  
child {node [b] at (0.0, 1.0) {  }  
}
}
;
}}

\newcommand{\forestXC}{
\tikz[planar forest] {

\node [b] at (0.0, 0.0) {  } 
child {node [b] at (-0.5, 1.0) {  }  
}
child {node [b] at (0.5, 1.0) {  }  
child {node [b] at (0.0, 1.0) {  }  
}
}
;
}}

\newcommand{\forestYC}{
\tikz[planar forest] {

\node [b] at (0.0, 0.0) {  } 
;
}}

\newcommand{\forestAD}{
\tikz[planar forest] {

\node [b] at (0.0, 0.0) {  } 
child {node [b] at (-0.5, 1.0) {  }  
}
child {node [b] at (0.5, 1.0) {  }  
}
;
}}

\newcommand{\forestBD}{
\tikz[planar forest] {

\node [b] at (0.0, 0.0) {  } 
;
}}

\newcommand{\forestCD}{
\tikz[planar forest] {

\node [b] at (0.0, 0.0) {  } 
child {node [b] at (0.0, 1.0) {  }  
child {node [b] at (0.0, 1.0) {  }  
}
}
;
}}

\newcommand{\forestDD}{
\tikz[planar forest] {

\node [b] at (0.0, 0.0) {  } 
;
}}

\newcommand{\forestED}{
\tikz[planar forest] {

\node [b] at (0.0, 0.0) {  } 
child {node [b] at (0.0, 1.0) {  }  
}
;
}}

\newcommand{\forestFD}{
\tikz[planar forest] {

\node [b] at (0.0, 0.0) {  } 
child {node [b] at (0.0, 1.0) {  }  
}
;
}}

\newcommand{\forestGD}{
\tikz[planar forest] {

\node [b] at (0.0, 0.0) {  } 
child {node [b] at (0.0, 1.0) {  }  
}
;
}}

\newcommand{\forestHD}{
\tikz[planar forest] {

\node [b] at (0.0, 0.0) {  } 
;
}}

\newcommand{\forestID}{
\tikz[planar forest] {

\node [b] at (0.0, 0.0) {  } 
child {node [b] at (0.0, 1.0) {  }  
}
;
}}

\newcommand{\forestJD}{
\tikz[planar forest] {

\node [b] at (0.0, 0.0) {  } 
;
}}

\newcommand{\forestKD}{
\tikz[planar forest] {

\node [b] at (0.0, 0.0) {  } 
child {node [b] at (-0.5, 1.0) {  }  
}
child {node [b] at (0.5, 1.0) {  }  
child {node [b] at (0.0, 1.0) {  }  
}
}
;
}}

\newcommand{\forestLD}{
\tikz[planar forest] {

\node [b,label={ [label distance=-1mm]0:{ \scriptsize 1 } }] at (0.0, 0.0) {  } 
;
}}

\newcommand{\forestMD}{
\tikz[planar forest] {

\node [b,label={ [label distance=-1mm]0:{ \scriptsize d } }] at (0.0, 0.0) {  } 
;
}}

\newcommand{\forestND}{
\tikz[planar forest] {

\node [b] at (0.0, 0.0) {  } 
;
}}

\newcommand{\forestOD}{
\tikz[planar forest] {

\node [w] at (0.0, 0.0) {  } 
;
}}

\newcommand{\forestPD}{
\tikz[planar forest] {

\node [b] at (0.0, 0.0) {  } 
;
}}

\newcommand{\forestQD}{
\tikz[planar forest] {

\node [w] at (0.0, 0.0) {  } 
;
}}

\newcommand{\forestRD}{
\tikz[planar forest] {

\node [w] at (0.0, 0.0) {  } 
;
}}

\newcommand{\forestSD}{
\tikz[planar forest] {

\node [b] at (0.0, 0.0) {  } 
;
}}

\newcommand{\forestTD}{
\tikz[planar forest] {

\node [w] at (0.0, 0.0) {  } 
;
}}

\newcommand{\forestUD}{
\tikz[planar forest] {

\node [b] at (0.0, 0.0) {  } 
;
}}

\newcommand{\forestVD}{
\tikz[planar forest] {

\node [w] at (0.0, 0.0) {  } 
;
}}

\newcommand{\forestWD}{
\tikz[planar forest] {

\node [b] at (0.0, 0.0) {  } 
;
}}

\newcommand{\forestXD}{
\tikz[planar forest] {

\node [w] at (0.0, 0.0) {  } 
;
}}

\newcommand{\forestYD}{
\tikz[planar forest] {

\node [b] at (0.0, 0.0) {  } 
;
}}

\newcommand{\forestAE}{
\tikz[planar forest] {

\node [b] at (0.0, 0.0) {  } 
child {node [w] at (0.0, 1.0) {  }  
}
;
}}

\newcommand{\forestBE}{
\tikz[planar forest] {

\node [b] at (0.0, 0.0) {  } 
child {node [w] at (-0.5, 1.0) {  }  
}
child {node [w] at (0.5, 1.0) {  }  
}
;
}}

\newcommand{\forestCE}{
\tikz[planar forest] {

\node [b] at (0.0, 0.0) {  } 
child {node [w] at (-1.0, 1.0) {  }  
}
child {node [w] at (0.0, 1.0) {  }  
}
child {node [w] at (1.0, 1.0) {  }  
}
;
}}

\newcommand{\forestDE}{
\tikz[planar forest] {

\node [b] at (0.0, 0.0) {  } 
child {node [w] at (-1.5, 1.0) {  }  
}
child {node [w] at (-0.5, 1.0) {  }  
}
child {node [w] at (0.5, 1.0) {  }  
}
child {node [w] at (1.5, 1.0) {  }  
}
;
}}

\newcommand{\forestEE}{
\tikz[planar forest] {

\node [b,label={ [label distance=-1mm]0:{ \scriptsize i } }] at (0.0, 0.0) {  } 
child {node [b,label={ [label distance=-1mm]180:{ \scriptsize j } }] at (-0.5, 1.0) {  }  
}
child {node [b,label={ [label distance=-1mm]0:{ \scriptsize k } }] at (0.5, 1.0) {  }  
child {node [b,label={ [label distance=-1mm]180:{ \scriptsize l } }] at (-0.5, 1.0) {  }  
}
child {node [b,label={ [label distance=-1mm]0:{ \scriptsize m } }] at (0.5, 1.0) {  }  
}
}
;
}}

\newcommand{\forestFE}{
\tikz[planar forest] {

\node [b] at (0.0, 0.0) {  } 
;
}}

\newcommand{\forestGE}{
\tikz[planar forest] {

\node [b] at (0.0, 0.0) {  } 
child {node [b] at (-0.5, 1.0) {  }  
}
child {node [b] at (0.5, 1.0) {  }  
}
;
}}

\newcommand{\forestHE}{
\tikz[planar forest] {

\node [b] at (0.0, 0.0) {  } 
child {node [b] at (0.0, 1.0) {  }  
child {node [b] at (-0.5, 1.0) {  }  
}
child {node [b] at (0.5, 1.0) {  }  
}
}
;
}}

\newcommand{\forestIE}{
\tikz[planar forest] {

\node [b] at (0.0, 0.0) {  } 
child {node [b] at (-1.0, 1.0) {  }  
}
child {node [b] at (0.0, 1.0) {  }  
}
child {node [b] at (1.0, 1.0) {  }  
}
;
}}

\newcommand{\forestJE}{
\tikz[planar forest] {

\node [b] at (0.0, 0.0) {  } 
child {node [b] at (0.0, 1.0) {  }  
child {node [b] at (0.0, 1.0) {  }  
child {node [b] at (-0.5, 1.0) {  }  
}
child {node [b] at (0.5, 1.0) {  }  
}
}
}
;
}}

\newcommand{\forestKE}{
\tikz[planar forest] {

\node [b] at (0.0, 0.0) {  } 
child {node [b] at (0.0, 1.0) {  }  
child {node [b] at (-1.0, 1.0) {  }  
}
child {node [b] at (0.0, 1.0) {  }  
}
child {node [b] at (1.0, 1.0) {  }  
}
}
;
}}

\newcommand{\forestLE}{
\tikz[planar forest] {

\node [b] at (0.0, 0.0) {  } 
child {node [b] at (-0.5, 1.0) {  }  
}
child {node [b] at (0.5, 1.0) {  }  
child {node [b] at (-0.5, 1.0) {  }  
}
child {node [b] at (0.5, 1.0) {  }  
}
}
;
}}

\newcommand{\forestME}{
\tikz[planar forest] {

\node [b] at (0.0, 0.0) {  } 
child {node [b] at (-1.5, 1.0) {  }  
}
child {node [b] at (-0.5, 1.0) {  }  
}
child {node [b] at (0.5, 1.0) {  }  
}
child {node [b] at (1.5, 1.0) {  }  
}
;
}}

\newcommand{\forestNE}{
\tikz[planar forest] {

\node [b] at (0.0, 0.0) {  } 
;
}}

\newcommand{\forestOE}{
\tikz[planar forest] {

\node [b] at (0.0, 0.0) {  } 
;
}}

\newcommand{\forestPE}{
\tikz[planar forest] {

\node [b] at (0.0, 0.0) {  } 
;
}}

\newcommand{\Char}{\text{Char}}

\newcommand{\Graph}{\text{Graph}}

\newcommand{\smallO}[1]{\ensuremath{\mathop{}\mathopen{}o\mathopen{}\left(#1\right)}}

\newcommand{\func}[1]{\mathcal{C}^{\infty}(#1)}

\newcommand{\groupoid}{\mathcal{G} \rightrightarrows M}

\allowdisplaybreaks

\newtheorem{theorem}{Theorem}[section]
\newtheorem{definition}[theorem]{Definition}
\newtheorem*{definition*}{Definition}

\newtheorem{proposition}[theorem]{Proposition}
\newtheorem{lemma}[theorem]{Lemma}
\newtheorem{remark}[theorem]{Remark}
\newtheorem*{remark*}{Remark}

\newtheorem*{remarks*}{Remarks}
\newtheorem{corollary}[theorem]{Corollary}

\newtheorem*{notation*}{Notation}
\newtheorem{ex}[theorem]{Example}
\newtheorem*{ex*}{Example}

\newtheorem*{exs*}{Examples}

\newtheorem*{app*}{Application}
\newtheorem{conjecture*}{Conjecture}
\def\ts{\thinspace}

\title{
Butcher series for Hamiltonian Poisson integrators through symplectic groupoids
}

\author{
Adrien Busnot Laurent\textsuperscript{1} and Oscar Cosserat\textsuperscript{2}
}

\date{\vspace{-1cm}}

\begin{document}
\footnotetext[1]{
Univ Rennes, INRIA (Research team MINGuS), IRMAR (CNRS UMR 6625) and ENS Rennes,
France.
Adrien.Busnot-Laurent@inria.fr.}
\footnotetext[2]{
Göttingen Mathematisches Institut,
Georg-August-Universität Göttingen,
Office 021, Hauptgebäude,
Bunsenstraße 3-5,
37073 Göttingen - Germany.
Oscar.Cosserat@mathematik.uni-goettingen.de
}

\maketitle

\begin{abstract}

We exhibit a new pre-Lie algebra in the framework of symplectic groupoids and, in turn, use Butcher trees to compute the Taylor series of a solution of a Hamilton-Jacobi equation on any symplectic groupoid $\groupoid.$ The impact of this new algebraic approach is twofold. On the geometric side, it yields algebraic operations to approximate Lagrangian bisections of $\GG$ using the Butcher-Connes-Kreimer Hopf algebra and, in turn, aims at a better understanding of the group of Hamiltonian diffeomorphisms of $M.$ On the computational side, we define a new class of Poisson integrators for Hamiltonian dynamics on Poisson manifolds.

\smallskip


\smallskip

\end{abstract}


\section{Introduction}

\subsection{Context}

A Poisson bracket on a smooth manifold $M$ equips the space of smooth functions of this manifold $\func{M}$ with a Lie algebra structure $(\func{M}, \{.,.\})$. Therefore, it is natural to ask about the existence of a Lie group integrating it. In the context of Poisson manifolds, there exists an extremely profitable approach to this question: instead of looking for an infinite-dimensional Lie group, one constructs a finite-dimensional Lie groupoid $\mathcal{G}$ over $M$. This Lie groupoid turns out to have a natural symplectic structure. Therefore, symplectic Lie groupoids are the global counterpart to Poisson structures. They encode in particular three different aspects of Poisson geometry: foliation theory (the partition of any Poisson manifold into leaves), symplectic geometry (the geometry along any leaf) and Lie theory. Concerning the question of integrating the Lie algebra of smooth functions, there exists a group object keeping track of this integration inside the symplectic groupoid: the group of Lagrangian bisections. A major interest of symplectic groupoids in mechanics is the deep relation between Lagrangian bisections of the Lie groupoid $\mathcal{G}$ and Hamiltonian dynamics on $M.$

Another interest of Lagrangian bisections lies in mathematical physics purposes. In \cite{Weinstein1987}, a formal correspondence is spelled between symplectic groupoids and $C^{*}$-algebra theory, where Lagrangian submanifolds are the elements of the non-commutative algebra. The groupoid inverse corresponds to the conjugation and the product law corresponds to the tensor product. There, Lagrangian bisections are unitary elements. At about the same time, the works \cite{Karasev1987, zakrzewski1990} started a research program on deformation theory and the quantisation of Poisson manifold through symplectic groupoids. This brought in turn a considerable attention on the topic \cite{weinstein1991, dSW, karabegov2005, hawkins2008}.

Lagrangian deformations rely heavily on combinatorial tools, with some analogous to the ones of the order analysis in geometric numerical integration. A relation between deformation of symplectic groupoids and high-order Runge-Kutta numerical methods has been explained in \cite{Cattaneo2005} and formulated in terms of operads in \cite{Cattaneo2010}. Symplectic realisations are constructed using the Butcher group in \cite{dherin2016}, while \cite{cabrera2020} gave a detailed construction of local symplectic groupoids using Butcher series of Hamilton-Jacobi generating functions. Here, the relation between mathematical physics and numerical analysis lies at the comparison between the deformation parameter $\hbar$ in the mathematical physics context and the time-step $\Delta t$ in numerical purposes. An analogy seems to hold between both situations: we expect tools from one field to be fruitfully applied to the other one and the other way around.

Symplectic groupoids extend beyond their use in mathematical physics, and they have been used in particular for computational purposes. Indeed, the relation between Hamiltonian dynamics on $M$ and Lagrangian bisections has been applied to the numerical approximation of Hamiltonian flows on Poisson manifolds \cite{oscar}. The idea appears first in \cite{Ge1990}, while the case of fiberwise linear Poisson structures has been studied in \cite{deLeon2017}.  More precisely, given any Hamiltonian $H \in \func{M},$ a Hamilton-Jacobi equation is used to relate its Hamiltonian flow $(\phi_t^H)_t$ to a smooth family of Lagrangian bisections $(L_t)_t,$ provided that $t$ is small enough. A truncation at any order of the solution of this Hamilton-Jacobi equation allows one to recover the initially considered Hamiltonian dynamics in an approximate way, and this approximation has been proved to be numerically satisfactory when it is compared to traditional methods \cite{Oscar2022}. To construct these numerical methods, a crucial use has been made of the deformation of Lagrangian submanifolds.

At the intersection between mathematical physics and numerical analysis, a rich combinatorial formalism has been developed since the sixties. Butcher series were first introduced in~\cite{Butcher72aat,Hairer74otb} (see also~\cite{Hairer06gni,McLachlan17bsa,Butcher21bsa}) for the study of order conditions for Runge-Kutta methods in numerical analysis.
They were later applied successfully to a variety of different fields such as geometric numerical integration~\cite{Iserles00lgm, Hairer06gni}, quantum field theory~\cite{Connes98har}, rough paths~\cite{Gubinelli10ror, Hairer15gvn}, or stochastic numerics \cite{Burrage96hso, Laurent20eab, Laurent21ocf, Busnot26osr}.

In light of the interplay between mathematical physics, numerical analysis and Poisson geometry, it is natural to look for the combinatorics of the approximation of Hamiltonian dynamics on a Poisson manifold. The modern approach to such algebraic formalism relies extensively on Hopf algebras \cite{Chartier10aso, Lundervold11hao, Bronasco22cef, Bronasco25hoi}. In this article, we identify, through the use of bi-realisations, the pre-Lie and Hopf algebra structures appearing in the context of geometric numerical integration of Hamiltonian systems on Poisson manifolds. Furthermore, we uncover the algebraic and computational consequences.

\subsection{Content of the paper}

In Section \ref{sec:preli_mais_pas_pre_lie}, we recall how the Hamiltonian dynamics of $H \in \func{M}$ is recovered by Lagrangian bisections of a symplectic groupoid $\GG$ of the Poisson manifold $M$ through a Hamilton-Jacobi equation. Jets are used to introduce the involved groups and to deduce an approximation of the Hamiltonian dynamics at any arbitrary order. In Section \ref{sec:pre_lie}, we provide a new pre-Lie combinatorial formalism to compute the formal solution of this Hamilton-Jacobi equation. We give two more applications of this pre-Lie algebra formalism. First, we use in Section \ref{sec:comp} the Butcher-Connes-Kreimer Hopf algebra to provide new injective group morphisms from the characters of the Butcher-Connes-Kreimer Hopf algebra to the group of jets of smooth families of bisections at the identity section (see Theorem \ref{thm:somewhat}). In Section \ref{sec:num}, we explain how this new algebraic formalism applies to high order approximations of Lagrangian bisections using Runge-Kutta numerical methods, which deliver as a by-product new Hamiltonian Poisson integrators.

\section{Preliminaries}\label{sec:preli_mais_pas_pre_lie}

\subsection{Basics on Poisson geometry}

In this section, we give a concise summary of various notions of Poisson geometry. We do not intend to give any introduction of the topic. Instead, the reader may consult \cite{Marle1987} about Hamiltonian dynamics and symplectic geometry and \cite{Crainic2021, Zung2005} about Poisson structures and symplectic groupoids.

Let $(M, \{.,.\})$ be a Poisson manifold. In this article, we are interested in Hamiltonian dynamics on $M:$ for any Hamiltonian $H \in \func{M},$ we study the differential equation on $M$
\begin{equation}\label{eq:Ham_eq}
\dot{x}(t) = X_H \bigl(x(t)\bigr)
\end{equation}
where $X_H \colon f \in \func{M} \mapsto \{H,f\} \in \func{M}$ is the Hamiltonian vector field\footnote{We denote derivations of $\func{M}$ and vector fields the same way.} of $H.$ Since one of the main motivations of the present work is the construction of new numerical methods, let us recall the notion of Hamiltonian Poisson integrator.

\begin{definition}[\cite{Oscar2022}]
A Hamiltonian Poisson integrator for the Hamiltonian $H \in \func{M}$ is a family of diffeomorphisms $\varphi_{t} \colon M \to M,$ $t \in I$ a small real parameter, with the following property: there exists a time-dependent Hamiltonian $(\tilde{H}_t)_{t \in I} \in \func{M \times I}$ such that $\varphi$ is the time-dependent Hamiltonian flow of $\tilde{H}.$ The Hamiltonian Poisson integrator is said to be of order $k$ if for any test function $f \in \func{M},$
\begin{equation}\label{eq:cond_order}
\forall \; 0 \leq i \leq k, \; \frac{\partial^i (f \circ \varphi_t)}{\partial t^i}_{|t=0}  = \frac{\partial^i (f \circ \phi_t^H)}{\partial t^i}_{|t=0}.
\end{equation}
\end{definition}

Symplectic methods are an important particular case of Hamiltonian Poisson integrators and are a major motivation for this work. Another remark is that all that matters for $(\tilde{H}_t)_{t \in I}$ is its existence. A sufficient condition on $(\tilde{H}_t)_{t \in I}$ for its time-dependent flow to provide a Hamiltonian Poisson integrator of order $k$ for $H$ is that it is an approximation of the Hamiltonian $H$ of the same order as the integrator: $\tilde{H}_t = H + \smallO{t^k}.$ As we can see in the local order condition \eqref{eq:cond_order}, Taylor series with respect to the time $t$ play an important role in our context to count the order of approximation of a dynamics.

\begin{remark}[Global error]
The question of global error for symplectic and Poisson integrators led to the development of backward error analysis \cite{Benettin1994} and is an active field of research \cite{dambrosio2024, cosserat2025}.
\end{remark}

We introduce now a geometric space used to construct Hamiltonian Poisson integrators. To any Poisson manifold $(M, \{.,.\})$ is associated a local symplectic groupoid $\groupoid$ over $M$ (see \cite{Weinstein1987, Karasev1987}). We write $\alpha \colon \mathcal{G} \to M$ and $\beta \colon \mathcal{G} \to M$ for the source and target maps respectively. The tubular neighborhood theorem of \cite{Weinstein1971} provides a local model around the identity section of $\groupoid$ to realize $\mathcal{G}$ as a neighborhood of the zero section inside $T^*M.$ With a slight abuse of notation, we keep the same letters: the tubular neighborhood is called $\mathcal{G}$ and $\alpha$ and $\beta$ denote again the resulting maps from $\mathcal{G}$ to $M.$

\begin{theorem}[\cite{Karasev1987, Crainic2011}]
There exists a tubular neighborhood $\mathcal{G} \subset T^*M$ of the zero section of $T^*M$ and two surjective submersions $\alpha$ and $\beta$ from $\mathcal{G}$ to $M$ such that
\begin{enumerate}
\item $\alpha \circ 0 = \beta \circ 0 = \text{Id}_M,$ where $0 \colon M \to T^*M$ is the zero section of the vector bundle $T^*M,$
\item $\alpha$ is a Poisson morphism and $\beta$ is an anti-Poisson morphism, where $\mathcal{G}$ is equipped with $\{.,.\}_{\omega}$ the Poisson bracket of the canonical symplectic form on $\mathcal{G} \subset T^*M,$
\item $\alpha$ and $\beta$ have symplectically orthogonal fibers: $\forall \; f,g \in \func{M}, \; \{ \alpha^*f, \beta^*g\}_{\omega} = 0.$ 
\end{enumerate}
\end{theorem}
This realisation of the local symplectic groupoid inside $T^*M$ was named \textit{bi-realisation} in \cite{oscar}, because $\alpha$ provides a symplectic realisation of the Poisson manifold $M$ while the maps $(\alpha, \beta)$ are a bi-submersion \cite{androulidakis2009} of the singular foliation on $M$ induced by the Poisson structure. The zero section is the identity for the groupoid product. We emphasize that the symplectic form of $\groupoid$ becomes the canonical symplectic form $\omega.$  $\GG$ is therefore a Poisson manifold endowed with the Poisson bracket $\{.,.\}_{\omega}.$ We also recall that the cotangent projection $\pi \colon \mathcal{G} \twoheadrightarrow M$ is in general different from the structural maps $\alpha$ and $\beta.$ 
\begin{remark}
A bi-realisation such that the groupoid inverse is the multiplication by $-1$ on each cotangent fiber is said to be "symmetric" in \cite[rk 2.19]{Cattaneo2010}. Following \cite{Karasev1987}, \cite{cabrera2020} gave a construction of a symmetric bi-realisation for any Poisson structure in coordinates. Note that for any Poisson manifold, there exist bi-realisations that are not symmetric.
\end{remark}

\subsection{Hamilton-Jacobi equation for Hamiltonian Poisson integrators}

In this section, we recall after \cite{oscar} how a Hamilton-Jacobi equation allows one to lift up Hamiltonian dynamics on $M$ to a bi-realisation $\mathcal{G}$ by describing Hamiltonian flows in terms of generating functions.

First, let us consider a dynamics that is a bit more general than Hamiltonian dynamics. Let $\theta \in \Omega^1(M)$ be a $1$-form. Denoting by $\varpi \in \Gamma(\bigwedge^2 TM)$ the bivector field of the Poisson bracket $\{.,.\},$ we write $X_{\theta} = \varpi(\theta, \cdot) \in \mathfrak{X}(M)$ for the vector field generated by $\theta$ and $\phi_t^{\theta}$ the flow of $X_{\theta}$ at time $t.$ In the sequel, we always assume flows to be complete. We recall the following standard properties \cite[Chap.\ts III]{Marle1987}.

\begin{proposition}\label{prop:1form_flow}
\begin{enumerate}
\item For any $x \in M$ and for any time $t,$ $\phi_t^{\theta}(x)$ belongs to the same symplectic leaf as $x.$
\item Let us assume that $\theta$ is closed. Then, $\phi_t^{\theta}$ is a Poisson automorphism that admits any symplectic leaf as an invariant set and preserves $\theta.$ In equations, denoting by $\mathcal{F}_x$ the symplectic leaf of $x \in M,$
\[
\forall x \in M, \phi_t^{\theta}(x) \in \mathcal{F}_x \text{ and }  (\phi_t^{\theta})_{*} \pi = \pi \text{ and } (\phi_t^{\theta})^* \theta = \theta.
\]
\end{enumerate}
\end{proposition}

This classical result justifies the notion of Hamiltonian Poisson integrator. Indeed, following the flow of a time-dependent Hamiltonian guarantees to stay on a symplectic leaf and to preserve the Poisson structure.

Let us set the case of general 1-forms on $M$ aside and from now on, we assume the $1$-form $\theta \in \Omega_0^1(M)$ closed. We state here the main theorem of this section.

\begin{theorem}[Hamilton-Jacobi equation on a local symplectic groupoid, \cite{oscar}]\label{thm:Ham_jac}
Let $\theta \in \Omega_0^1(M),$ $I$ be a small open interval containing $0$ and $(\zeta_t)_{t \in I} \in \Omega^1(M)^{I}$. For any $t \in I$, set 
\begin{equation}\label{eq:Lt_zetat}
L_t = \Graph(\zeta_t) \subset \GG
\end{equation}
and assume $\alpha_{|L_t} \colon L_t \rightarrow M$ to be a diffeomorphism for all $t \in I$. Set
\begin{equation}\label{eq:diff_bislag}
\varphi_t = \beta \circ (\alpha_{|L_t})^{-1} \colon M \to M.
\end{equation}
Then, $\forall t \in I, \; \varphi_t = \phi^{\theta}_t$ if and only if
\begin{equation}
\label{eq:Ham_jac_1form_intial0}
\zeta_0 = 0 \text{ and } \forall t \in I, \; 
\frac{\partial \zeta_t}{\partial t} = (\zeta_t)^*\alpha^*\theta.
\end{equation}
\end{theorem}

Before carrying on, we specify a terminology: in Hamilton-Jacobi theory, it is classical \cite[Sec. 7.9]{Marsden1998} to denote by the graph of a section of the cotangent bundle the actual image of the corresponding map, justifying what we mean by the equation \eqref{eq:Lt_zetat}.

\subsection{Jets of smooth families of bisections}\label{sec:jets_bislag}

In this article, we will develop tools to approximate solutions of the Hamilton-Jacobi equation, e.g. \eqref{eq:Ham_jac_1form_intial0}, at high order with respect to the variable $t$. For this precise reason, we will need an appropriate notion of jets. In this section, we thus explain some geometry of the previous Hamilton-Jacobi equation using jets and Taylor series.

\begin{remark}
The literature already contains a notion of jets of a bisection (\cite[Ex. 3.1.8]{cattafi2020}, \cite[Sec. A.1.1]{crainic2017}), although not quite the one we need for our purpose: in the present work, we deal with families of bisections, parametrized by the time-step of the numerical method. Therefore, the notion of jets of smooth families of bisections is new.
\end{remark}

It will be useful in the sequel to keep in mind an important property of the graph $L_t$ of $\zeta_t$ for small $t \in I$: it is a Lagrangian bisection. Let us explain this now.

Let us first explain what a bisection is (see also \cite[Sec.\ts 15.2]{dSW}). In any Lie groupoid $G \rightrightarrows M$, a bisection is a submanifold $L$ of $G$ such that the source and target restrict to $L$ as a diffeomorphism.

\begin{proposition}[\cite{dSW}]\label{prop:prod}
The set of bisections of $G \rightrightarrows M$ forms a group. 
\end{proposition}

\begin{proof}
Let $L_1$, $L_2$ two bisections. The product is set to be
\begin{equation*}
L_1 \circ L_2 = \{ a \cdot b \in G, \; a \in L_1, b \in L_2, \beta(a) = \alpha (b) \}.
\end{equation*}
The associativity and the existence of an inverse follow from the axioms of a groupoid law. The neutral element is given by the identity section.
\end{proof}

In our local groupoid $\GG$, since $L_0$ is the zero section, the graph $L_t$ of $\zeta_t$ is thought to be transverse to the fibers of $\alpha$ and $\beta$, at least locally and for small $t$. The submanifold $L_t$ of Eq. \eqref{eq:Lt_zetat} is thus a bisection.

First, in our smooth setting, we need a notion of family of bisections, all being close to the identity section. They can be understood as smooth perturbations of the identity section.
\begin{definition}[Smooth family of bisections]
We denote by smooth family of bisections of $\GG$ the following data:
\begin{itemize}
\item a real open interval $I$ containing $0,$
\item a family $L = (L_t)_{t \in I}$ of bisections of $\GG,$ where $L_0 =0$ is the image of the zero section and the surjective map $\{ (x,t) \in \GG \times I, x \in L_t \} \twoheadrightarrow I$ is a submersion.
\end{itemize}
\end{definition}
\begin{remark}\label{ex:smooth_bis}
Since the fibers of $\alpha$ are transverse to the zero section, a generic example of smooth family of bisections of $\GG$ is provided by any smooth family of $1$-forms $(\zeta_t)_{t \in I} \in \Omega^1(M)$ for some small interval $I.$
\end{remark}

Now, we introduce a notion of $\infty$-jets for such objects. To achieve this, let $f \in \func{\GG}$ be a test function and $L$ a smooth family of bisections. For any $t \in I,$ let us set $\psi_t = \left(\alpha_{|L_t}\right)^{-1} \colon M \overset{\sim}{\to} L_t.$ We consider the Taylor series at $t=0$ of $f \circ \psi_t  \in \func{M \times I}$. This provides a map
\[
\mathcal{J}^L \colon 
\begin{array}{ccc}
\func{\GG} & \to & \func{M}\bigl[[t]\bigr]\\
f & \mapsto & f \circ \psi_t
\end{array}
\]
where $f \circ \psi_t \in \func{M}\bigl[[t]\bigr]$ stands for its Taylor series $\sum\limits_{j = 0}^{\infty} \frac{t^j}{j!} \frac{\partial^j f \circ \psi_t}{\partial t^j}_{| t=0}$ at $t=0.$
\begin{definition}[$\infty$-jets of smooth families of bisections of $\GG$]\label{def:jets_bis}
The map $\mathcal{J}^L$ is said to be the $\infty$-jet of the smooth family of bisections $L.$ In the sequel, we denote by $\mathbb{B}$ the space of such maps: 
\[
\mathbb{B} = \{ \mathcal{J}^L \colon \func{\GG} \to \func{M}\bigl[ [t] \bigr], \text{ L smooth family of bisections} \}.
\]
\end{definition}

$\mathbb{B}$ is a space of equivalence classes of smooth families of bisections.

\begin{ex}\label{ex:jet_bis}
Following remark \ref{ex:smooth_bis}, if $L = (\Graph(\zeta_t))_{t \in I},$ the data of the $\infty$-jet of the smooth family of bisections $L$ is equivalent to the one of the Taylor series of $(\zeta_t)_{t \in I}$ with respect to $t$ at $t=0.$ With a slight abuse of terminology, we will then write that the jet of $L$ equals the Taylor series of $(\zeta_t)_{t \in I}$ at $t=0.$
\end{ex}

In the following, one naturally defines a product on $\mathbb{B}$. Let $L^1 = (L^1_t)_{t \in I}$ and $L^2 = (L^2_t)_{t \in J}$ two smooth families of Lagrangian bisections. Locally on $\GG,$ there exists $t_0 >0$ such that for $|t|<t_0,$ the product $L_t^1 \cdot L_t^2$ of Prop. \ref{prop:prod} is defined in the local symplectic groupoid $\GG.$ Then, we set the product to be the pointwise product with respect to the real infinitesimal parameter $t:$
\[
\mathcal{J}^{L^1_t} \cdot \mathcal{J}^{L^2_t} \colon 
\begin{array}{ccc}
\func{\GG} & \to & \func{M}\bigl[[t]\bigr]\\
f & \mapsto & f \circ \left(\alpha_{|L_t^1 \cdot L_t^2 }\right)^{-1} 
\end{array}.
\]

The following property is a straightforward consequence of Definition \ref{def:jets_bis} and is left to the reader.
\begin{proposition}
$\mathbb{B}$ is a group, with the neutral element being the jet constantly equal to the identity section:
\[
\mathcal{J}^{\text{Id}} \colon 
\begin{array}{ccc}
\func{\GG} & \to & \func{M}\bigl[[t]\bigr]\\
f & \mapsto & f \circ 0
\end{array}.
\]
\end{proposition}

The reader may also keep in mind from \cite[Chap. II]{Weinstein1987} that for any two $L_1$ and $L_2$ bisections of a groupoid, denoting by $L_1 \cdot L_2$ the bisection being the product of $L_1$ and $L_2$, the induced diffeomorphisms on the base verify
\begin{equation}\label{eq:comp_bislag}
\left(\beta \circ (\alpha_{|L_1})^{-1}\right) \circ \left(\beta \circ (\alpha_{|L_2})^{-1}\right) = \beta \circ (\alpha_{|L_1 \cdot L_2})^{-1}.
\end{equation}
This group morphism from the group of bisections to the diffeomorphisms of $M$ will play a role all along this article, see, e.g., Corollary \ref{cor:iso} at the end of the present section, and Theorem \ref{thm:somewhat} later on.

Up to now, we only discussed the property for the graph $L_t \subset \GG$ of $\zeta_t \in \Omega_0^1(M)$ to be a bisection. Let us now remark a second property of the bisection $L_t$ by adding the symplectic geometry up. Since $\zeta_t$ is closed, its graph $L_t$ is a Lagrangian submanifold in $\GG$ for the canonical symplectic form of $T^*M$. This leads us to consider the space of jets of smooth families of Lagrangian bisections. We denote it by $\overline{\mathbb{ L}}.$ Since translations by Lagrangian bisections in $\GG$ are symplectomorphisms from $\GG$ to $\GG$, $\overline{\mathbb{L}}$ carries a natural structure.
\begin{proposition}
$\overline{\mathbb{ L}}$ is a subgroup of $\mathbb{B}$.
\end{proposition}

In our context, the correspondence spelled by the Hamilton-Jacobi equation in Theorem \ref{thm:Ham_jac} can be interpreted as the direct relation between the group $\overline{\mathbb{ L}}$ and the dynamics generated by closed 1-forms on the base. Let us be more precise. Since the bisections $(L_t)_t$ are Lagrangian and close to the zero section of $T^*M,$ the induced Poisson diffeomorphisms on the base manifold $M$ $\beta \circ \left( \alpha_{| L_t} \right)^{-1}$ are flows of time-dependent closed forms. It follows from Proposition \ref{prop:1form_flow} that these Poisson diffeomorphisms stay on a leaf of the symplectic foliation. Furthermore, as explained in the following remark, these closed forms are exact if and only if the Lagrangian bisections are graphs of exact one-forms. 

\begin{remark}[Generating functions]\label{rem:gen_func}
The closedness of $\theta \in \Omega^1(M)$ is equivalent to the one of $\zeta_t \in \Omega^1(M)$ for all $t \in I.$ The same equivalence holds about exactness and leads us to Hamiltonian dynamics. Let us assume $\theta$ to be exact and $H \in \func{M}$ a Hamiltonian being a primitive of $\theta.$ As a consequence, there exists $S \in \func{M \times I}$ such that $dS_t = \zeta_t.$ Equation \eqref{eq:Ham_jac_1form_intial0} becomes
\begin{equation}\label{eq:Ham_jac_exact}
\left\{
\begin{array}{cc}
\frac{\partial S_t}{\partial t} =& (dS_t)^*\alpha^*H + \chi(t)\\
dS_0 =& 0
\end{array}
\right.
\end{equation}
where $\chi \in \func{I}$ is an arbitrary time-dependent constant. In the following, we choose $\chi$ to be $0$ and $S_0=0.$ Using equation \eqref{eq:diff_bislag}, the graph of $dS$ recovers the Hamiltonian dynamics generated by $H.$ $S$ is thus said to be a generating function for $H.$
\end{remark}

After a classical terminology for generating functions, a Lagrangian bisection is called \emph{exact} if it is the graph of an exact one form. The jets of exact Lagrangian bisections form a group again.
\begin{proposition}\label{prop:subgroup_exact}
We set 
\[
\mathbb{L}  = \{ \mathcal{J}^L \colon \func{G} \to \func{M} \bigl[ [t] \bigr], L \text{ smooth family of exact Lagrangian bisections} \}.
\]
Then, $\mathbb{L}$ is a subgroup of $\overline{\mathbb{L}}.$
\end{proposition}

\begin{proof}
Let $\mathcal{J}^{L^1}, \mathcal{J}^{L^2} \colon \func{\GG} \to \func{M}\bigl[ [t] \bigr]$ two jets of exact Lagrangian bisections. We show that $\mathcal{J}^{L^1 \cdot L^2}$ is a jet of exact Lagrangian bisections. Using \cite[Prop. 2.19]{oscar} and the remark \ref{rem:gen_func}, there exists two time-dependent Hamiltonians $\tilde{H}^1$, $\tilde{H}^2 \in \func{M \times I}$ such that for any test function $f \in \func{\GG},$ $\mathcal{J}^{L^1}(f) = f \circ \phi_t^{\alpha^* \tilde{H}_t^1} \circ 0$ and $\mathcal{J}^{L^2}(f) = f \circ \phi_t^{\alpha^* \tilde{H}_t^2} \circ 0.$ Now, the composition of two time-dependent Hamiltonian flows is again a time-dependent Hamiltonian flow:
\begin{align*}
\mathcal{J}^{L^1 \cdot L^2}(f) &= f \circ \phi_t^{\alpha^* \tilde{H}_t^2} \circ \phi_t^{\alpha^* \tilde{H}_t^1} \circ 0\\
			       &= f \circ \phi_t^{\alpha^* \tilde{H}_t} \circ 0,
\end{align*}
where $\tilde{H}_t = \tilde{H}_t^2 + \tilde{H}_t^1 \circ \phi_t^{\tilde{H}_t^2}.$ The same computation proves the existence of an inverse. Its neutral element is clearly the jet coming from the smooth family being constantly equal to the identity section.
\end{proof}

As already mentioned, the importance of Lagrangian bisections in mechanics is due to their relation with Hamiltonian dynamics. We define the analog in our context of the Hamiltonian group of, e.g., \cite[Def. 1.11.]{Crainic2021}, and of the group of locally Hamiltonian diffeomorphisms.

\begin{definition}[$\infty$-jets Hamiltonian group]
We call $\mathcal{H}$ the group of $\infty$-jets of pull-backs of time-dependent Hamiltonian flows:
\[
\mathcal{H} = \bigl\{ \mathcal{F}^{\tilde{H}_t} \colon 
\begin{array}{ccc}
\func{M} & \to & \func{M}\bigl[[t]\bigr]\\
f & \mapsto & f \circ \phi^{\tilde{H}_t}_t
\end{array}, \tilde{H} \in \func{M \times I} \bigr\}. 
\]
and $\overline{\mathcal{H}}$ the group of $\infty$-jets of pull-backs of flows generated by time-dependent closed $1$-forms, sometimes called locally Hamiltonian flows:
\[
\overline{\mathcal{H}} = \bigl\{ \mathcal{F} \colon 
\begin{array}{ccc}
\func{M} & \to & \func{M}\bigl[[t]\bigr]\\
f & \mapsto & f \circ \phi^{\tilde{\theta}_t}_t
\end{array}, \tilde{\theta} \in \Omega_0^1(M)^I \bigr\}. 
\]
where, as before, $f \circ \phi^{\tilde{H}_t}_t$ and $f \circ \phi^{\tilde{\theta}_t}_t$ stand for their Taylor series with respect to $t.$
\end{definition}

By adding these definitions to the remark \ref{rem:gen_func} on exact Lagrangian bisections, we obtain the following corollary of Theorem \ref{thm:Ham_jac}. The proof relies on the same interpolation technique as the one of Proposition \ref{prop:subgroup_exact} and is left as an exercise.

\begin{corollary}\label{cor:iso}
There is a surjective group morphism from $\overline{\mathbb{L}}$ to $\overline{\mathcal{H}}$.

This morphism restricts to a surjective group morphism from $\mathbb{L}$ to $\mathcal{H}$.
\end{corollary}

Those morphisms are in general not injective, as illustrated by the example of the Poisson bracket being identically zero.

\subsection{Approximations in the group of jets of Lagrangian bisections}\label{sec:approx_bisec}

One other consequence of Theorem \ref{thm:Ham_jac} is the following corollary, yielding the existence of approximations of $\phi^{\theta}_t$, for $\theta \in \Omega_0^1(M)$ any closed $1$-form, at arbitrary order that preserve the Poisson geometry in the sense of Proposition \ref{prop:1form_flow}.

\begin{corollary}\label{cor:Ham_jac_k}
Let $k \in \mathbb{N}$ and $(\zeta_t^{k})_{t \in I} \in \Omega^1_0(M)^{I}.$ Set $(L_t^{k})_{t \in I} = \bigl(\Graph(\zeta_t^{k}) \bigr)_{t \in I}$ and
\[
\varphi_t^{k} = \beta \circ (\alpha_{|L_t^{k}})^{-1}, \; t \in I.
\]
If for all $t \in I,$ 
\[
\left\{
\begin{array}{cc}
\frac{\partial \zeta_t^k}{\partial t} =& (\zeta_t^k)^*\alpha^*\theta + \smallO{t^k}\\
\zeta_0 =& 0
\end{array}
\right.,
\]
then $(\varphi_t^k)_{t \in I}$ is the flow of a time-dependent closed $1$-form $(\theta^k_t)_{t \in I} \in \Omega^1_0(M)$ such that $\theta^k_t = \theta + \smallO{t^k}.$ In particular, $(\varphi_t^k)_{t \in I}$ is an approximation of $(\phi^{\theta}_t)_{t \in I}$ at order at least $k$ such that
\begin{itemize}
\item for any $x \in M, \varphi_t^k(x)$ and $x$ belong to the same symplectic leaf of $(M, \{.,.\}),$
\item $\varphi_t^k$ is a Poisson diffeomorphism for all $t \in I.$
\end{itemize}
\end{corollary}
The family of Lagrangian bisections $L^k = (L_t^{k})_{t \in I}$ is an approximation at order $k$ of the family of Lagrangian bisections $L = (L_t)_{t \in I}$ given by equation \eqref{eq:Lt_zetat}. The jet formalism of Section \ref{sec:jets_bislag} provides a rigorous framework.
\begin{proposition}
With the notations of Corollary \ref{cor:Ham_jac_k}:
\[
\forall f \in \func{\GG}, \; \mathcal{J}^L(f) = \mathcal{J}^{L^k}(f) + \smallO{t^k}.
\]
\end{proposition}

Using Corollary \ref{cor:Ham_jac_k}, this Hamilton-Jacobi equation has been applied to computational mechanics in \cite{Oscar2022}: the purpose was to truncate solutions of this equation to approximate Hamiltonian dynamics on the base. Indeed, if $S_t^k$ is a solution of equation \eqref{eq:Ham_jac_exact} at order $k,$ then the map
\[
\phi_{\Delta t}^k = \beta \circ \left( \alpha_{| \Graph(\text{d} S_{\Delta t}^k)} \right)^{-1}
\]
is a Hamiltonian Poisson integrator for $H$ at order $k$ and time-step $\Delta t.$ This article is devoted to explain an algebraic formalism for the construction of high order Hamiltonian dynamics approximations. We also hope this to be of interest for a better understanding of the group of Lagrangian bisections of a symplectic groupoid and, in turn, a better understanding of the group of Poisson diffeomorphisms of a Poisson manifold. For that reason, let us consider an equation being a bit more general than \eqref{eq:Ham_jac_1form_intial0}. Namely, we consider the same one with a generic initial condition $\zeta_0 \in  \Omega^1_0(M)$ such that its graph belongs to $\GG:$
\begin{equation}\label{eq:Ham_jac_1form}
\left\{
\begin{array}{cc}
\frac{\partial \zeta_t}{\partial t} =& (\zeta_t)^*\alpha^*\theta\\
\zeta_0 \in&  \Omega^1_0(M)
\end{array}
\right.
\end{equation}
Regarding equation \eqref{eq:comp_bislag}, equation \eqref{eq:Ham_jac_1form} is of interest while one looks at composition of Hamiltonian flows on $M.$ We will come back to that in Section \ref{sec:comp}.

Since we aim at understanding high order approximations of Hamiltonian dynamics on $M,$ let us study truncated solutions of \eqref{eq:Ham_jac_1form}. For this, we introduce the Lie-algebroid bracket $[.,.]$ on $\Omega^1(T^*M)$ defined by
\[
[ \zeta_1, \zeta_2] = \mathcal{L}_{X_{\zeta_1}}\zeta_2 - \mathcal{L}_{X_{\zeta_2}}\zeta_1 - d \left( \omega(X_{\zeta_1}, X_{\zeta_2}) \right)
\]
where $\omega$ is the canonical symplectic form of $T^*M$ and $X_{\zeta_i},$ $i = 1,2,$ are the vector fields on $T^*M$ generated by the canonical symplectic form out of the 1-forms $\zeta_i.$ Thanks to the following lemma, this Lie bracket allows us to compute approximations at arbitrary order of Hamilton-Jacobi equation. In the following and all along this article, $\pi \colon T^*M \twoheadrightarrow M$ denotes the cotangent projection.

\begin{lemma}\label{lemma: time_der}
Let $(\xi_t)_{t \in I} \in \Omega^1(T^* M)^I$ and $(\zeta_t)_{t \in I} \in \Omega^1(M)^I$. Then,
\begin{equation}\label{eq:der_forms}
\frac{\partial}{\partial t} \left((\zeta_t)^* \xi_t \right) = (\zeta_t)^* \left([\xi_t, \pi^* \frac{\partial \zeta_t}{\partial t}] +  \frac{\partial \xi_t}{\partial t}\right).
\end{equation}
Similarly, for any $(f_t)_{t \in I} \in \func{T^*M \times I}$ and $(S_t)_{t \in I} \in \func{M \times I}$,
\begin{equation}\label{eq:der_smoothfunctions}
\frac{\partial}{\partial t} \left((dS_t)^*f_t \right) = (dS_t)^* \left(\{f_t, \pi^* \frac{\partial S_t}{\partial t}\}_{\omega} + \frac{\partial f_t}{\partial t}\right).
\end{equation}
\end{lemma}

\begin{proof}
We start by proving equation \eqref{eq:der_smoothfunctions}. Let $x \in M.$ It follows from classical symplectic geometry of the cotangent bundle that the curve $\gamma \colon t \in I \mapsto d_x S_t \in T^*M$ is the flow of the time-dependent Hamiltonian vector field of $\pi^*\frac{\partial S_t}{\partial t}$ starting at $d_x S_0.$ Equation \eqref{eq:der_smoothfunctions} is then a plain consequence of the chain rule.
Since
\[
\forall f,g \in \func{T^*M}, \; d\{f,g\}_{\omega} = [df, dg],
\]
equation \eqref{eq:der_forms} is obtained by usual extension from smooth functions to exact 1-forms, and then from closed forms to generic 1-forms using the Leibniz rule. 
\end{proof}

Let us illustrate the use of Lemma \ref{lemma: time_der}. The algebroid bracket $[.,.]$ allows us to obtain iterated time derivatives of the solution of the Hamilton-Jacobi equation \eqref{eq:Ham_jac_1form}. We spell out the order $2:$
\begin{equation}\label{eq:der_order2}
\frac{\partial^2 \zeta_t}{\partial t^2} = (\zeta_t)^*[\alpha^*\theta, \pi^* (\zeta_t)^* \alpha^* \theta].
\end{equation}
After applying Lemma \ref{lemma: time_der} a second time, we obtain the order $3$:
\begin{equation}\label{eq:der_order3}
\frac{\partial^3 \zeta_t}{\partial t^3} = (\zeta_t)^*\left([[\alpha^*\theta, \pi^* (\zeta_t)^* \alpha^* \theta], \pi^* (\zeta_t)^*\alpha^* \theta] + [\alpha^*\theta, \pi^* (\zeta_t)^* [\alpha^* \theta, \pi^* (\zeta_t)^*\alpha^* \theta]]  \right).
\end{equation}

\section{A Pre-Lie approach to Hamiltonian Poisson integrators}\label{sec:pre_lie}

As $\GG \subset T^*M$, bisections of the local symplectic groupoid $\GG$ that are close to the identity section are described by graphs of time-dependent $1$-forms. In this section, we use the notions of jets developed in the last section to approximate bisections through formal solutions of Hamilton-Jacobi equations. In order to approximate a given bisection by $1$-forms, we introduce an appropriate space $J_\xi^\infty$ equipped with a pre-Lie algebra structure.  We then explain how this pre-Lie algebra encodes expansions of solutions of the Hamilton-Jacobi equation \eqref{eq:Ham_jac_1form} through the introduction of Butcher series. More precisely, Proposition \ref{proposition:HJButcher} explains how $J_\xi^\infty$ stands for the space of Taylor coefficients of formal solutions of the Hamilton-Jacobi equation.  In Section \ref{sec:simpl}, we study particular cases of bi-realisations where some algebraic simplifications arise if the initial condition is chosen to be zero, that is, if we start from the identity section of $\GG.$

\subsection{A Pre-Lie algebra for the Hamilton-Jacobi equation}

Let us consider the most general case of our Hamilton-Jacobi equation. Let $\xi\in \Omega^1(\GG)$ and let us denote by $\Omega^{1, \GG}(M)$ the subset $\{ \zeta \in \Omega^1(M), \Graph(\zeta) \subset \GG \}$ of the set of $1$-forms on $M.$ We are interested in the Taylor expansion of the general Hamilton-Jacobi flow
\begin{equation}
\label{equation:general_HJ}
\frac{\partial\zeta_t}{\partial t}=\zeta_t^*\xi,\quad \zeta_0\in \Omega^{1, \GG}(M).
\end{equation}
We set $\mathcal{E}$ the vector space of maps defined as follows:
\[
	\mathcal{E} = \Omega^1(M)^{ \Omega^{1, \GG}(M)}.
\]

\begin{ex}
$\xi$ defines an element of $\mathcal{E},$ that we write again $\xi,$ provided by the pull-back:
\[
\xi \colon \begin{array}{ccc}
\Omega^{1, \GG}(M) & \to & \Omega^1(M)\\
\zeta & \mapsto & \zeta^*\xi
\end{array}.
\]
Note that the map induced by $\xi$ is of first order in $\zeta$: it only uses the value of $\zeta$ and its differential. For our purpose, we need maps of higher order, and this is the reason why we work on $\mathcal{E}$ instead of only working with pull-backs of differential forms.
\end{ex}
We introduce the maps
\[
\eta_{0} \colon
\begin{array}{ccc}
\mathcal{E} & \to & \mathcal{E}\\
f & \mapsto & f
\end{array},
\]
and for all $k \in \mathbb{N}^*,$

\[
\eta_{k} \colon
\begin{array}{ccc}
\mathcal{E}^{\otimes k} & \to & \mathcal{E}\\
f_1, \ldots, f_k & \mapsto & \Bigl[ \zeta \in \Omega^{1, \GG}(M) \mapsto  \zeta^*[\ldots [ \xi, \pi^* f_k (\zeta)], \ldots, \pi^* f_1(\zeta)] \in \Omega^1(M) \Bigr]
\end{array}.
\]
We consider the smallest real vector subspace $J_\xi^\infty \subset \mathcal{E}$ fulfilling the following conditions:

\begin{enumerate} 
\item $\xi \in J_\xi^\infty,$
\item for all $k \in \mathbb{N}^*,$ the image of the restriction of the map $\eta_{k}$ to $(J_\xi^\infty)^{\otimes k}$ is a subset of $J_\xi^\infty.$
\end{enumerate}

Some of those maps appeared at the end of the section \ref{sec:approx_bisec}: the reader may compare, e.g., $\eta_1$ with Equation \eqref{eq:der_order2}.

\begin{remark}\label{rem:k_homo}
The space $J_\xi^\infty$ admits in particular the elements $\xi, \eta_1(\xi), \eta_2(\xi, \xi), \ldots,$ and real linear combinations of such elements. Note that $\eta_k(\xi, \ldots, \xi)$ is $k+1$-homogeneous in $\xi.$
\end{remark}
The space $J_\xi^\infty$ is naturally equipped with the linear product $\triangleright$ defined as follows. For any $h \in J_\xi^\infty,$
\[
h\triangleright \xi = \eta_1(h)
\]
and we extend $\triangleright$ to an inner product on $J_\xi^\infty$ by linearity and the following Leibniz rule: for any $k \in \mathbb{N}^*,$ $h, f_1, \ldots, f_k  \in J_\xi^\infty,$
\[
h\triangleright \Bigl(\eta_{k}(f_1, \ldots, f_{k}) \Bigr)=\sum\limits_{i=1}^k \eta_{k}(f_1, \ldots, h \triangleright f_i, \ldots, f_k) + \eta_{k+1}(h, f_1, \ldots f_k).
\]
Let us first observe the following.
\begin{lemma}\label{lemma:eta_k}
For all $k \in \mathbb{N}^*,$ $1 \leq i,j \leq k$ and $h_1, \ldots, h_k \in J_\xi^\infty,$
\[
\eta_k(h_1, \ldots, h_i, \ldots, h_j, \ldots, h_k) = \eta_k(h_1, \ldots, h_j, \ldots, h_i, \ldots, h_k)
\]
\end{lemma}

\begin{proof}
The proof follows from the equality
\[
\forall \gamma_1, \gamma_2 \in \func{\GG}, \; \forall \zeta \in \Omega^{1, \GG}(M), \;  \{\pi^*\zeta^*\gamma_1,\pi^*\zeta^*\gamma_2\}_{\omega}=0,
\]
and the Jacobi identity.
\end{proof}

Let us now state an important property of the product $\triangleright.$ Namely, it endows $J_\xi^\infty$ with a pre-Lie algebra algebraic structure.
In the case of linear Poisson structures, the product $\triangleright$ is strongly related to the pre-Lie Magnus expansion \cite{Ebrahimi11tda, Casas19ccv} (see \cite{Curry20tme} for the post-Lie-Magnus extension). To the best of our knowledge, the extension of this pre-Lie structure to a general Poisson manifold is a new result.

\begin{proposition}
\label{proposition:post-Lie}
The space $(J_\xi^\infty,\triangleright)$ is a pre-Lie algebra, that is, for all $f,g,h\in J_\xi^\infty$,
\begin{align}
\label{equation:pre-Lie}
(f\triangleright g)\triangleright h-f\triangleright (g\triangleright h)=(g\triangleright f)\triangleright h-g\triangleright (f\triangleright h).
\end{align}
\end{proposition}

\begin{proof}
The proof is a direct computation using the construction of the space $J_\xi^\infty.$ Let $f,g \in J_\xi^\infty.$ We obtain $\eta_2(f,g) = \eta_2(g,f)$ from Lemma \ref{lemma:eta_k}, which yields equation \eqref{equation:pre-Lie} in the case $h=\xi \in J_\xi^\infty.$
By induction, let $k \in \mathbb{N}^*$ and $h_1, \ldots, h_k$ such that for all $1 \leq i \leq k,$
\[
(f\triangleright g)\triangleright h_i-f\triangleright (g\triangleright h_i)=(g\triangleright f)\triangleright h_i-g\triangleright (f\triangleright h_i).
\]
Let $h=\eta_k(h_1, \ldots, h_k) \in J_\xi^\infty.$ Then, a calculation yields
\begin{align*}
(f\triangleright g)\triangleright h-f\triangleright (g\triangleright h)
&=-\eta_{k+2}(f, g, h_1, \ldots,  h_k)
\\
&-\sum\limits_{i=1}^k (\eta_{k+1}(g, h_1, \ldots, f \triangleright h_i, \ldots, h_k) + \eta_{k+1}(f, h_1, \ldots, g \triangleright h_i, \ldots, h_k))\\
&-\sum\limits_{\substack{1 \leq i,j \leq k \\ i \neq j} } \eta_k(h_1, \ldots, f \triangleright h_j, \ldots, g \triangleright h_i, \ldots, h_k)\\
&+\sum\limits_{i=1}^k \eta_k(h_1, \ldots, (f \triangleright g) \triangleright h_i-f \triangleright (g \triangleright h_i) , \ldots, h_k)
\end{align*}
The induction hypothesis and Lemma \ref{lemma:eta_k} yield the result.
\end{proof}

Lemma \ref{lemma: time_der} provides an interpretation of the product $\triangleright$ as a variational derivation in the sense of \cite{Olver93aol}: for instance,
\[ 
\bigl( \xi \triangleright \xi \bigr)(\zeta) = \zeta^*[\xi, \pi^*\zeta^*\xi],
\]
and this recalls again Equation \eqref{eq:der_order2}.

In that framework, considering jet spaces to iterate derivatives is therefore a natural idea. Let us also notice that the definition of $\triangleright$ -- as well as its definition domain $J_\xi^\infty$ -- is tied to the choice of the form $\xi \in \Omega^1(\GG)$. For instance, we raise the following remark.
\begin{remark}\label{rem:eval_morphism}
Given two spaces $(J_{\xi_1}^\infty,\triangleright_1)$ and $(J_{\xi_2}^\infty,\triangleright_2)$ and a map $\varphi\colon J_{\xi_1}^\infty \rightarrow J_{\xi_2}^\infty$ satisfying
\[
\forall i \in \mathbb{N}, \; \forall (f_j)_{i \geq j \geq 1} \in (J_{\xi_1}^{\infty})^i, \; \varphi\Bigl(\eta_{i}\bigl(f_1, \ldots, f_i \bigr)\Bigr)=\eta_{i}\Bigl(\varphi(f_1), \ldots, \varphi(f_i)  \Bigr),
\]
one can show the following: $\varphi$ is a morphism -- that is $\varphi(f\triangleright_1 g)= \varphi(f)\triangleright_2 \varphi(g)$ -- if and only if $\varphi(\xi_1)=\xi_2$.
\end{remark}

Hamiltonian Poisson integrators rely on efficient discretisations of the exact Hamilton-Jacobi equation \eqref{eq:Ham_jac_exact}. As stated in Theorem \ref{thm:Ham_jac}, the Lagrangian bisections are represented by the solution of Hamilton-Jacobi equations of the general form provided by equation \eqref{eq:Ham_jac_1form}.
The present pre-Lie formalism conveniently gives an explicit expression of the Taylor expansion of the solution of equation \eqref{eq:Ham_jac_1form}, as stated in the following proposition.
\begin{proposition}
\label{proposition:HJ_expansion}
For $\xi\in \Omega^1(\GG)$, the formal solution of the Hamilton-Jacobi equation \eqref{equation:general_HJ} satisfies
\begin{align}
\zeta_t&=\exp^\triangleright(t\xi)(\zeta_0) \nonumber\\
\label{equation:postLieexpansion}
&= \sum_{n=0}^\infty \frac{t^n}{n!}\xi^{\triangleright n}(\zeta_0)\\
&=\zeta_0+t(\zeta_0)^*\xi+\frac{t^2}{2}\Bigl(\xi\triangleright \xi\Bigr)(\zeta_0)+\frac{t^3}{3!} \Bigl( \xi\triangleright(\xi\triangleright \xi) \Bigr)(\zeta_0)+\frac{t^4}{4!} \Bigl(\xi\triangleright(\xi\triangleright(\xi\triangleright \xi))\Bigr)(\zeta_0)+\dots, \nonumber
\end{align}
with the convention $\xi^{\triangleright 0}=\id$.
\end{proposition}

Proposition \ref{proposition:HJ_expansion} is proven straightforwardly by induction, in the spirit of \cite{Lundervold11hao}. The use of bi-realisations allows one to derive the flow $\zeta_t$ that writes as a standard pre-Lie expansion, similarly to the case of linear Poisson structures \cite{Casas19ccv}, while the context of Hamiltonian dynamics on a generic Poisson manifold is significantly more intricate.

\subsection{Butcher series expansion of a solution of Hamilton-Jacobi equation}

The expansion \eqref{equation:postLieexpansion} is concise and simple as each order has only one Taylor term. However, we are left with the computation of the iterations of $\triangleright$. In this section, we further expand the Taylor expansion of \eqref{equation:general_HJ}, relying on a pre-Lie formalism of Butcher trees.

We set $T$ the set of Butcher trees, recursively defined by 
\[
\forestA\in T, \quad (\tau_n,\dots ,\tau_1)_{\forestB}\in T, \quad \tau_1,\dots,\tau_n\in T,
\]
where the root is graphically represented at the bottom. $(\tau_n,\dots ,\tau_1)_{\forestC}$ denotes the tree with the root ${\forestD}$ and the $n$ trees $\tau_n,\dots ,\tau_1$ plugged to the root. These trees are non-planar, meaning that the order of the branches does not matter: for instance, $\forestE=\forestF$.
We denote by $\TT=\Span_\R(T)$ the real vector space generated by $T$. The grafting of trees $\curvearrowright$ is defined as a product on $T$ returning the sum of all possibilities (counted with multiplicity) of grafting the root of one tree on the nodes of another tree. For instance:
\[
\forestG\curvearrowright\forestH=\forestI+\forestJ,\quad
\forestK\curvearrowright\forestL=\forestM+2\forestN,\quad
\forestO\curvearrowright\forestP=\forestQ.
\]
By extending $\curvearrowright$ linearly on $\TT$, this defines the pre-Lie algebra $(\TT,\curvearrowright)$ of Butcher trees.
We anticipate on Section \ref{sec:comp} by defining the empty forest $\textbf{1}$ and denote $\TT_{\textbf{1}}=\TT\oplus \R\textbf{1}$.
A natural grading on $\TT_{\textbf{1}}$ is given by the number of nodes: $\abs{\forestR}=3$, $\abs{\textbf{1}}=0$.

The translation between the geometric structure $(J_\xi^\infty,\triangleright)$ and the algebraic structure $(\TT,\curvearrowright)$ is obtained through the elementary differential map. 
\begin{definition}[Elementary differential map]
Let $\xi\in \Omega^1(\GG).$ The elementary differential map $\F^\xi \colon \TT_{\textbf{1}} \to \EE$ associated to $\xi$ is the linear map defined by $\F^\xi(\textbf{1})= \id$, $\F^\xi(\forestS)= \xi$, and
\[
\F^\xi \Bigl( (\tau_n,\dots ,\tau_1)_{\forestT} \Bigr) ( \zeta )=\zeta^* [[\dots[\xi,\pi^*  \F^\xi \bigl(\tau_1 \bigr)(\zeta)],\dots],\pi^* \F^\xi \bigl( \tau_n \bigr) (\zeta)], \quad \zeta\in \Omega^1(M).
\]
\end{definition}
The following result is a straightforward consequence of the definition of the product $\triangleright$.
\begin{proposition}
The elementary differential $\F^\xi\colon \TT \rightarrow J_\xi^\infty$ is a pre-Lie algebra morphism:
\[
\F^\xi(\tau_2\curvearrowright \tau_1)=\F^\xi(\tau_2)\triangleright \F^\xi(\tau_1).
\]
\end{proposition}
This morphism allows one to transport Proposition \ref{proposition:HJ_expansion} and to rewrite it naturally in terms of trees.
Let $(\zeta_t)_{t}$ be the solution of the Hamilton-Jacobi equation \eqref{equation:general_HJ}, then its expansion satisfies
\begin{align}
\nonumber
\zeta_t&=\F^\xi \Bigl(\exp^\curvearrowright(t\forestU)\Bigr)(\zeta_0)\\
&=\F^\xi\Big(\textbf{1}+t\forestV+\frac{t^2}{2}\forestW\curvearrowright \forestX+\frac{t^3}{3!}\forestY\curvearrowright(\forestAB\curvearrowright \forestBB)+\frac{t^4}{4!}\forestCB\curvearrowright(\forestDB\curvearrowright(\forestEB\curvearrowright \forestFB))+\dots\Big)(\zeta_0)\\
&= \sum_{n=0}^\infty \frac{t^n}{n!}\F^\xi \Bigl(\forestGB^{\curvearrowright n}\Bigr)(\zeta_0),\nonumber
\end{align}
with the convention $\tau^{\curvearrowright 0}=\textbf{1}$.

\begin{remark}[$J_\xi^\infty$ and jet spaces]
\label{rem:relation_Rn}
The infinite jet bundle is the infinite-dimensional vector bundle over $M$ whose coordinates are the partial derivatives at any order of the components of a smooth vector field $f\in \mathfrak{X}(M)$ (see \cite{Anderson89tvb, Olver93aol, Lee13its}).
It typically allows one to write coordinate-free Taylor expansions of flows, and thus is heavily used for the analysis of Runge-Kutta and Lie-group methods \cite{Iserles00lgm, Hairer06gni}.
The standard approach with Butcher series allows one to represent specific Taylor expansions in a jet space, obtained by repeated compositions of the covariant derivation operator \cite{Lundervold11hao}.
We follow here a similar approach by fixing a one form $\xi \in \Omega^1(\GG)$ and considering the space $J_\xi^\infty$ on which one can iterate compositions of the operator $\triangleright$.
For $\xi = \alpha^* \theta,$ $\theta \in \Omega^1(M),$ we will use these flows to construct jets of smooth families of bisections.
\end{remark}

Let us now use trees to encode explicitly the expansion $\exp^\curvearrowright.$ The appropriate concept for such formal expansions is the one of Butcher series, often called B-series. 
\begin{definition}[\cite{Butcher21bsa}]
The B-series associated to $\xi \in \Omega^1(\GG)$ is the following formal power series indexed by a coefficient map $a\in\TT^*$:
\[
B^\xi \colon
\begin{array}{ccc}
\TT_{\textbf{1}}^* & \to & \EE\bigl[ [t] \bigr]\\
a & \mapsto & a(\textbf{1})\id+\sum_{\tau\in T} \frac{a(\tau)}{\sigma(\tau)} \F^{t\xi}(\tau)
\end{array}
\]
where $\sigma(\tau)$ is the number of graph automorphisms of $\tau$, also called the symmetry coefficient.
\end{definition}
We refer to \cite[Sec.\ts III]{Hairer06gni} for an explicit formula of the symmetry coefficient. Note that if $\tau$ is a tree with $k$ nodes, $\F^{t\xi}(\tau)$ will be homogeneous in $t$ of degree $k$ (see also the remark \ref{rem:k_homo}). Then, the Butcher series of the solution of a Hamilton-Jacobi equation is the following.
\begin{proposition}
\label{proposition:HJButcher}
Let $(\zeta_t)_t \in (\Omega^1(M))^I$ be the solution of the Hamilton-Jacobi equation \eqref{equation:general_HJ}, then its Taylor expansion is given by
\begin{equation}\label{eq:HJButcher}
B^{\xi}\bigl(e\bigr) (\zeta_0) \in \Omega^1(M) \bigl[ [t] \bigr],
\end{equation}
where $e \in \TT_{\textbf{1}}^*$ is given by
\begin{equation}\label{eq:e}
e(\textbf{1})=1,\quad e(\forestHB)=1,\quad e(\tau)=\frac{1}{\abs{\tau}}e(\tau_1)\dots e(\tau_n),\quad \tau= (\tau_n,\dots ,\tau_1)_{\forestIB}.
\end{equation}
\end{proposition}

Using Section \ref{sec:jets_bislag}, this proposition is interpreted in terms of jets. Let us assume that $\xi = \alpha^* \theta$ for some $\theta \in \Omega^1(M)$ and $(\zeta_t)_{t \in I} \in \Omega^1(M)^I$ the solution of the equation \eqref{eq:Ham_jac_1form}. Then, the jet of the smooth family of bisections $(\Graph(\zeta_t))_{t \in I}$ is given by the series $B^{\xi}(e)(\zeta_0)$ in $\Omega^1(M) \bigl[ [t] \bigr].$ Proposition \ref{proposition:HJButcher} also tells us what $J_\xi^\infty$ stands for. Namely, it is the space of Taylor coefficients of a formal solution of the Hamilton-Jacobi equation.

The representation of the Hamilton-Jacobi flow with trees allows us to conveniently provide an explicit expansion of $\zeta_t$ at any order. These calculations are called Farmer series in \cite[Sec. 3.3.4]{Cosserat23tac}, where Hamiltonian Poisson integrators were implemented. Our algebraic formalism greatly simplifies the tedious calculations written there. Using our construction of the appropriate Butcher series provided by Proposition \ref{proposition:HJButcher}, we find directly
\begin{align*}
\zeta_t&=\F^\xi\Big(
\textbf{1}
+t\forestJB
+\frac{t^2}{2}\forestKB
+\frac{t^3}{3!}\big(\forestLB+\forestMB\big)
+\frac{t^4}{4!}\big(\forestNB+\forestOB+3\forestPB+\forestQB\big)
+\dots\Big) (\zeta_0)\\
&=\zeta_0 +  \zeta_0^*\Big(
t\xi
+\frac{t^2}{2}[\xi,\pi^*\zeta_0^*\xi]
+\frac{t^3}{3!}\big([\xi,\pi^*\zeta_0^*[\xi,\pi^*\zeta_0^*\xi]]+[[\xi,\pi^*\zeta_0^*\xi],\pi^*\zeta_0^*\xi]\big)\\
&+\frac{t^4}{4!}\big([\xi,\pi^*\zeta_0^*[\xi,\pi^*\zeta_0^*[\xi,\pi^*\zeta_0^*\xi]]]+[\xi,\pi^*\zeta_0^*[[\xi,\pi^*\zeta_0^*\xi],\pi^*\zeta_0^*\xi]]\\
&+ 3[[\xi,\pi^*\zeta_0^*[\xi,\pi^*\zeta_0^*\xi]],\pi^*\zeta_0^*\xi]+[[[\xi,\pi^*\zeta_0^*\xi],\pi^*\zeta_0^*\xi],\pi^*\zeta_0^*\xi]\big)\\
&+\dots\Big).
\end{align*}
Note that we recover in particular the equations \eqref{eq:der_order2} and \eqref{eq:der_order3}.

\subsection{Occurrence of simplifications}
\label{sec:simpl}
As explicit formulas for bi-realisations are rarely available, numerical approximations may rely on approximate bi-realisations, such as the Karasev bi-realisation \cite{Karasev1987, cabrera2024}.
Such construction brings simplifications in our Taylor expansions, for which we provide a detailed description here.

Following \cite{Cosserat23tac}, let $\theta \in \Omega^1(M)$, $\xi=\alpha^* \theta$ and $\zeta_0=0$. In this section, we assume further the following property for the source map $\alpha$:
\begin{equation}
\label{eq:degen}
 0^*[\alpha^* \theta,\pi^* \theta]=0.
\end{equation}

\begin{remark}
Let us pick cotangent coordinates $(q,p)$ on $T^*M$. The simplification \eqref{eq:degen} occurs, for instance, when the differential of $\alpha$ along fibers of $T^*M$ at zero $(\partial_{p_j} \alpha_i (q, 0))_{1 \leq i,j \leq n}$ is anti-symmetric. In the Karasev bi-realisation, for all $1 \leq i,j \leq n,$ $\partial_{p_j} \alpha_i (q, 0) = \{ q_i, q_j \}.$ An interested reader might consult \cite[eq. (3.6)]{dherin2016} for more insights on this particular case.
\end{remark}
In the light of superconvergence (see, e.g., \cite{jay1996}), this simplification implies that the differential associated to some specific trees vanishes.
\begin{theorem}
\label{thm:superconvergence}
Let $T_0 \subset T$ the subset of trees that do not have a descendant of the form $\forestRB.$
Let $(\zeta_t)_t \in (\Omega^1(M))^I$ be the solution of the Hamilton-Jacobi equation \eqref{eq:Ham_jac_1form_intial0} under the simplification condition \eqref{eq:degen}, with initial condition $\zeta_0=0$. Then, its Taylor expansion is given by the formal series with coefficients being closed forms
\[
B^{  \alpha^* \theta}\bigl(e_0 \bigr) (0)= \sum_{\tau\in T_0} \frac{e(\tau)}{\sigma(\tau)} \F^{t \alpha^* \theta}(\tau)(0)\in \Omega^1_0(M) \bigl[ [t] \bigr],
\]
where $e_0 \in \TT_{\textbf{1}}^*$ is given by $e_0(\tau)=e(\tau) \ind_{T_0}$ and $e$ satisfies \eqref{eq:e}.
\end{theorem}

We illustrate Theorem \ref{thm:superconvergence} in the case $\theta = d H$. Indeed, we are now able to compute straightforwardly the Farmer series of \cite{Cosserat23tac} at high order. Let $(S_t)_t \in \func{M \times I}$ be a generating function of a Hamiltonian $H \in \func{M}$ provided by equation \eqref{eq:Ham_jac_exact}. Let us set $\zeta_t = \text{d}S_t$ for $t \in I.$ We write the expansion of $\zeta$ up to order $5:$
\vskip-3ex
\begin{align}
\zeta_t&=\F^{\alpha^* dH}\Big(
\textbf{1}
+t\forestSB
+\frac{t^3}{3!}\forestTB
+\frac{t^4}{4!}\big(\forestUB+\forestVB\big)
+\frac{t^5}{5!}\big(\forestWB+\forestXB+4\forestYB+\forestAC\big)
+\dots\Big)(0) \nonumber \\
\label{equation:true_expansion_HJ}
&=t\text{d}H
+\frac{t^3}{3!} 0^*[[\alpha^* \text{d}H,\pi^*  \text{d}H],\pi^*  \text{d}H] \nonumber \\&
+\frac{t^4}{4!}\Big(0^*[\alpha^* \text{d}H,\pi^*0^*[[\alpha^* \text{d}H,\pi^*  \text{d}H],\pi^*  \text{d}H]] +0^*[[[\alpha^* \text{d}H,\pi^*  \text{d}H],\pi^*  \text{d}H],\pi^*  \text{d}H]\Big) \\&
+\frac{t^5}{5!}\Big(
0^*[\alpha^* \text{d}H,\pi^*0^*[\alpha^* \text{d}H,\pi^*0^*[[\alpha^* \text{d}H,\pi^*  \text{d}H],\pi^*  \text{d}H]]] \nonumber \\&
+0^*[\alpha^* \text{d}H,\pi^*0^*[[[\alpha^* \text{d}H,\pi^*  \text{d}H],\pi^*  \text{d}H],\pi^*  \text{d}H]] \nonumber \\&
+4\cdot 0^*[[\alpha^* \text{d}H,\pi^*  0^*[[\alpha^* \text{d}H,\pi^*  \text{d}H],\pi^*  \text{d}H]],\pi^*  \text{d}H] \nonumber \\&
+0^*[[[[\alpha^* \text{d}H,\pi^*  \text{d}H],\pi^*  \text{d}H],\pi^*  \text{d}H],\pi^*  \text{d}H]\Big)
+\dots \nonumber 
\end{align}

\section{Composition of Lagrangian bisections and B-series}\label{sec:comp}

In this section, we use Butcher series theory to provide a combinatorial description of the group law of jets of smooth families bisections. More precisely, we prove that the product of jets of smooth families of bisections is encoded by the composition of B-series and the Butcher-Connes-Kreimer Hopf algebra in the sense of Theorem \ref{thm:somewhat}. The occurence of this Hopf algebra in the framework of symplectic groupoids is new. Our analysis relies heavily on the various notions constructed in Section \ref{sec:jets_bislag}.

Let us consider the symmetric tensor algebra $(S(\TT),\cdot)$ over trees, that is the vector space spanned by forests, with its unit being the empty forest $\textbf{1}$.
Let the Butcher-Connes-Kreimer coproduct $\Delta_{BCK}\colon \TT\rightarrow\TT\otimes \TT$ be given by
\[
\Delta_{BCK}(\tau)=\sum_{s\subset \mathfrak{S}_\tau} (\tau\setminus s)\otimes s,
\]
where $\mathfrak{S}_\tau$ is the set of the subtrees of $\tau$ that contain the root, including the empty tree.
One finds for instance
\begin{align}
\Delta_{BCK}(\forestBC)&=
\textbf{1} \otimes \forestCC
+\forestDC \otimes \forestEC
+\forestFC \otimes \forestGC
+\forestHC \otimes \textbf{1} \nonumber\\
\label{equation:BCK_ex}
\Delta_{BCK}(\forestIC)&=
\textbf{1} \otimes \forestJC
+\forestKC \otimes \forestLC
+\forestMC \otimes \forestNC
+\forestOC \otimes \forestPC
+\forestQC \otimes \forestRC
+\forestSC \otimes \forestTC
+\forestUC \otimes \textbf{1}.
\end{align}
The coproduct is extended on $S(\TT)$ by multiplicativity:
\[
\Delta_{BCK}(\textbf{1})=\textbf{1}\otimes \textbf{1},\quad \Delta_{BCK}(\tau_1\cdots \tau_n)=\Delta_{BCK}(\tau_1)\cdots \Delta_{BCK}(\tau_n), \quad \tau_1, \dots, \tau_n \in \TT.\]
Then, it is well-known that $(S(\TT),\textbf{1},\cdot,\textbf{1}^*, \Delta_{BCK})$ yields the Butcher-Connes-Kreimer Hopf algebra (\cite{Connes98har}), used in particular to represent differential operators (see \cite{Foissyadh} and references therein).
Note that the pre-Lie product $\triangleright$ can be extended on $S(\TT)$ by the standard Guin-Oudom procedure \cite{Oudom08otl}.

We call character a linear form $a\in S(\TT)^*$ that satisfies $a(\tau_1\cdot\tau_2)=a(\tau_1) a(\tau_2)$ and $a(\textbf{1})=1$. Note that given $a\in \TT^*$, there is a unique way to extend $a$ as a character on $S(\TT)$. The space of characters is denoted $\Char(S(\TT))$.
The composition of two B-series is detailed by the BCK Hopf algebra \cite{Chartier10aso}. Note that it applies for any Taylor expansion over $J^\infty_\xi$, not just for the exact Hamilton-Jacobi flow.

\begin{proposition}
For any $a_1,$ $a_2 \in \Char(S(\TT)),$ we set
\[
a_1* a_2=\mu\circ (a_1\otimes a_2) \circ \Delta_{BCK},
\]
where $\mu$ is the multiplication of real numbers. Then, $(\Char(S(\TT)),*)$ is a group, called the Butcher group, with unit $\delta_{\forestVC}$.
\end{proposition}
From the example \eqref{equation:BCK_ex}, we find, for instance,
\begin{align*}
a_1* a_2(\forestWC)&=
a_2(\forestXC)
+a_1(\forestYC) a_2(\forestAD)
+a_1(\forestBD) a_2(\forestCD)
+a_1(\forestDD)^2 a_2(\forestED)\\&
+a_1(\forestFD) a_2(\forestGD)
+a_1(\forestHD) a_1(\forestID) a_2(\forestJD)
+a_1(\forestKD)
\end{align*}

The product on $G_B$ defines by duality a product on Butcher series: for any $\xi \in \Omega^1(\GG)$ and $a_1,$ $a_2 \in \TT^*,$ denoting by $B^\xi(a^1)$ and $B^\xi(a^2)$ their B-series respectively, their composition is defined by
\begin{equation}\label{eq:comp_Bseries}
B^\xi(a_1) \circ B^\xi(a_2) :=B^\xi(a_2* a_1).
\end{equation}
The following proposition is the analog of \cite[Thm 3.1]{Chartier10aso} in our context. We use the terminology introduced in the example \ref{ex:jet_bis}, meaning that for any $(\zeta_t)_{t \in I} \in \Omega^1(M)^I$ such that $\zeta_t \in \Omega^{1, \GG}(M)$ for all $t \in I,$ we identify the jet of the smooth family of bisections $\Bigl(\Graph(\zeta_t)\Bigr)_{t \in I}$ with the Taylor series of $(\zeta_t)_t$ with respect to $t.$
\begin{proposition}\label{prop:composition}
Let $\theta \in \Omega^1(M)$ and $a_1, a_2 \in \TT^*.$ We set $\mathcal{J}^{L^1} \in \mathbb{B}$ and $\mathcal{J}^{L^2} \in \mathbb{B}$ to be the jets of bisections corresponding to the graphs of $B^{ \alpha^* \theta}\bigl(a_1\bigr)$ and $B^{ \alpha^* \theta}\bigl(a_2\bigr)$ respectively. Then, the Butcher series, evaluated at $0 \in \Omega^{1, \GG}(M)$, $\Bigl(B^{ \alpha^* \theta}(a_1) \circ B^{ \alpha^* \theta}(a_2)\Bigr)(0)$ is the jet of bisections $\mathcal{J}^{L^1} \cdot \mathcal{J}^{L^2} \in \mathbb{B}.$
\end{proposition}

The following theorem relates this proposition with equation \eqref{eq:comp_bislag} and states that the Butcher group encodes the product of jets of bisections in a local symplectic groupoid.
\begin{theorem}\label{thm:somewhat}
Let $\xi \in \Omega^1(\GG)$ and the map $\Psi^\xi \colon G_B \rightarrow \Omega^1(M)\bigl[ [ t ] \bigr]$ be given by
\[
\Psi^\xi(a)=B^{\xi}\bigl(a\bigr)(0).
\]
Then:
\begin{itemize}
\item For any $\theta \in \Omega^1(M),$ $\Psi^{\alpha^*\theta} \colon (G_B,*) \hookrightarrow (\mathbb{B}, \cdot)$ is an injective group morphism.
\item For any $\theta \in \Omega_0^1(M),$ $\Psi^{\alpha^*\theta}\colon (G_B,*) \hookrightarrow (\overline{\mathbb{L}}, \cdot)$ is an injective group morphism.
\item For any $H \in \func{M},$ $\Psi^{\alpha^*dH}\colon (G_B,*) \hookrightarrow (\mathbb{L}, \cdot)$ is an injective group morphism.
\end{itemize}
\end{theorem}

\begin{proof}
The maps are well defined thanks to Section \ref{sec:preli_mais_pas_pre_lie}. The group morphism properties are provided by Proposition \ref{prop:composition} and their injectivity follows from the definition of a Butcher series. 
\end{proof}

For simplicity, we considered jet spaces generated by one $1$-form $\xi$ (see also remark \ref{rem:eval_morphism}) and our Butcher series are devoted to the approximation of one chosen smooth family of bisections. It is worth mentioning that one can extend the previous formalism for flows driven by a finite amount of forms $\xi_1$,\dots, $\xi_d \in \Omega^1(\GG)$ by using decorated nodes. We sketch the construction and illustrate it here for $d=2$.
The jet space of the section \ref{sec:pre_lie} becomes $J^\infty_{\xi_1,\dots,\xi_d} \subset \mathcal{E}$ and is represented by the algebra of decorated trees spanned by $\forestLD$, \dots, $\forestMD$. The previously described algebra adapts in this setting, in the spirit of P-series for partitioned problems \cite{Hairer06gni}.

For instance, let us consider $J^\infty_{\xi_1,\xi_2} \in \mathcal{E}$ and bi-coloured trees where $\forestND$ stands for $\xi_1 \in \Omega^1(\GG)$ and $\forestOD$ stands for $\xi_2 \in \Omega^1(\GG)$.
Let us compose the B-series $B^{\xi_1}(\delta_{\forestPD}) \in J^{\infty}_{\xi_1} \bigl[ [t] \bigr]$ and $B^{\xi_2}(\delta_{\forestQD}) \in J^{\infty}_{\xi_2}\bigl[ [t] \bigr]$. The composition $B^{\xi_2}(\delta_{\forestRD}) \circ B^{\xi_1}(\delta_{\forestSD}) \in J^{\infty}_{\xi_1, \xi_2}\bigl[ [t] \bigr]$ is computed with the BCK Hopf algebra as in the equation \eqref{eq:comp_Bseries} and we find
\begin{align*}
\Bigl(B^{\xi_2}(\delta_{\forestTD}) \circ B^{\xi_1}(\delta_{\forestUD})\Bigr)(\zeta_0)
&=B^{\xi_1,\xi_2}\Bigl(\delta_{\forestVD}*\delta_{\forestWD}\Bigr)(\zeta_0)\\
&=\F^{t\xi_1,t\xi_2}\Big(
\forestXD
+\forestYD
+\forestAE
+\frac{1}{2!}\forestBE
+\frac{1}{3!}\forestCE
+\frac{1}{4!}\forestDE
+\dots
\Big)(\zeta_0)\\
&=\zeta_0^*\Big(
t(\xi_1+\xi_2)
+t^2 [\xi_1,\pi^*\zeta_0^* \xi_2]
+\frac{t^3}{2!}[[\xi_1,\pi^*\zeta_0^* \xi_2],\pi^*\zeta_0^* \xi_2]
+\dots
\Big).
\end{align*}

\section{Numerical methods for Hamiltonian systems on Poisson manifolds}\label{sec:num}

Using the pre-Lie structure of Butcher trees, we present a new class of Hamiltonian Poisson integrators based on Taylor and Runge-Kutta discretisations. We discretize the dynamics of the equation \eqref{eq:Ham_eq} generated by a Hamiltonian $H \in \func{M}.$ Therefore, we are interested in this section in the previous constructions with $\xi = \alpha^* dH \in \Omega^1(\GG)$, the initial condition $\zeta_0 = 0 \in \Omega^{1, \GG}(M)$, with and without the degeneracy condition \eqref{eq:degen}.

We emphasize that the framework of pre-Lie algebras and Butcher series is a central tool for the numerical integration of ODEs in $\R^d$. The surprising effect of the bi-realisations is that they translate a Poisson geometry into a simpler geometry on $\Omega^1(M)$, where one can apply numerical tools similar to the Euclidean context.

\subsection{Hamiltonian Poisson integrators by truncation of Taylor series}

We want to obtain Hamiltonian Poisson integrators of arbitrary order $N.$ Let us denote $T^N$ (resp.\ts $T_0^N$) the subset of $T$ (resp.\ts $T_0$) containing trees of order at most $N$. Let $e^N(\tau)=e(\tau)\ind_{\tau\in T^N}$ and $e_0^N(\tau)=e(\tau)\ind_{\tau\in T_0^N}$ be restrictions of the coefficient maps of Proposition \ref{proposition:HJButcher}, that truncate the formal series of the previous sections into finite sums.
By specifying the formal parameter $t$ as a small enough real quantity -- that will be the time-step $\Delta t$ of the numerical method -- the B-series of the equation \eqref{eq:HJButcher} becomes a locally well-defined\footnote{We can get rid of the issue of constructing a globally well-defined $1$-form by considering the time-step as an infinitesimal parameter.} $1$-form: $B^{\alpha^* dH}\bigl( e^N \bigr) (0) \in \Omega^1(M).$ Evaluating this $1$-form at any $x \in M,$ we obtain a well-defined point $B^{\alpha^* dH}_{x}\bigl( e^N \bigr) (0) \in \GG$ (and analogously for $B^{\alpha^* dH}_{x}\bigl( e_0^N \bigr)$). Based on truncations of the expansion \eqref{equation:true_expansion_HJ} of the Hamilton-Jacobi flow \eqref{equation:general_HJ}, we construct numerical methods for this flow and, in turn, approximate the Hamiltonian dynamics on $M$. The following theorem is devoted to approximate the dynamics generated by a Hamiltonian $H \in \func{M}.$

\begin{theorem}
\label{thm:THP}
The following Taylor-Hamiltonian-Poisson integrator is of order $N$ for solving equation \eqref{eq:Ham_eq}:
\begin{align}
\label{eq:inter_pt1}
y_n&=\alpha\Bigl(B_{x_n}^{ \alpha^* dH}\bigl(e^N\bigr)(0)\Bigr),\\
\label{eq:inter_pt2}
y_{n+1}&=\beta\Bigl(B_{x_n}^{\alpha^* dH}\bigl(e^N\bigr)(0)\Bigr),
\end{align} 
where the formal parameter $t$ is specified small enough and becomes the timestep $\Delta t$ of the method,
\[
B^{ \alpha^* dH}\bigl(e^N\bigr)(0) \colon 
\begin{array}{ccc} 
M & \to & \GG \\
x & \mapsto & B_{x}^{ \alpha^* dH}\bigl(e^N\bigr)(0)
\end{array}
\]
is the map provided by the Butcher series associated to $\alpha^* dH$, and $x_n \in M$ denotes an intermediary point implicitly defined by equation \eqref{eq:inter_pt1}.
In addition, under the degeneracy condition \eqref{eq:degen}, replacing $e^N$ by $e_0^N$ in \eqref{eq:inter_pt1}-\eqref{eq:inter_pt2} yields a method of order $N$.
\end{theorem}

By their very constructions, these methods are Hamiltonian Poisson integrators: they follow the flow of some time-dependent Hamiltonian. Therefore, they stay on a symplectic leaf and preserve any Casimir all along a trajectory. The following method has been benchmarked in \cite[Sec. 5.2.]{Oscar2022} on a Lotka-Volterra system in a neighborhood of a singularity.

\begin{ex}[Euler method]
The simplest method is the Euler method, given by Theorem \ref{thm:THP} for $N=1$. It is of first order and the associated iteration is
\[
y_n = \alpha(\Delta t \cdot d_{x_n} H),\quad
y_{n+1} = \beta(\Delta t \cdot  d_{x_n} H).
\]
Under the degeneracy condition \eqref{eq:degen}, this method is of order $N=2$.
\end{ex}

In \cite{Oscar2022}, a Hamiltonian Poisson integrator of order $2$ is benchmarked on the rigid body dynamics. Theorem \ref{thm:THP} recovers the methods of the former reference and provides a general framework for extending these methods to higher order.

\subsection{Runge-Kutta approach to Hamiltonian Poisson integrators}

The number of trees in $\TT^N$ and $\TT_0^N$ blows up quickly as the order $N$ gets larger, which makes the Taylor approach computationally expensive and often unstable.
As a solution, we propose the following class of Runge-Kutta-Hamiltonian-Poisson (RKHP) approximations for the high-order approximation of equation \eqref{eq:Ham_eq}.
\begin{align}
Z_i&=\Delta t \sum_{1 \leq j<i  } a_{ij} \cdot (Z_j^*\alpha^*dH), \quad 1 \leq i \leq s \nonumber\\
\psi^{\Delta t \cdot \xi}&= \Delta t \sum_{i=1}^s b_i \cdot  (Z_i^*\alpha^*dH) \in d \func{M}, \nonumber\\
\label{equation:RK_HJ}
y_n&=\alpha\Bigl(\psi^{\Delta t \cdot \alpha^*dH}_{x_n} \Bigr) \in M\\
y_{n+1}&=\beta\Bigr(\psi^{\Delta t \cdot \alpha^*dH}_{x_n} \Bigr) \in M. \nonumber
\end{align}
The $a_{ij} \in \mathbb{R}$ and $b_i \in \mathbb{R},$ $1 \leq i,j \leq s,$ are the coefficients of the method and shall be chosen in order to reach high order of accuracy with a small number of intermediate stages $s \in \mathbb{N}$.
For simplicity, we consider explicit integrators, i.e., $a_{ij}=0$ for $j\geq i$.
The $(Z_i)_{1 \leq i \leq s} \in \Omega^{1, \GG}(M)$ are all exact $1$-forms on $M$ that are explicitly defined by the first equation line. For any $\Delta t$ small enough, note that the definition \eqref{equation:RK_HJ} relies on a map
\[
\psi^{\Delta t \cdot \alpha^*dH} \colon
\begin{array}{ccc}
M & \to & \GG \\
x & \mapsto & \psi^{\Delta t \cdot \alpha^*dH}_x \colon =\Delta t \cdot \sum_{i=1}^s b_i \cdot  \bigl(d_x Z_i^*\alpha^* H\bigr)
\end{array}
\]
The coefficients are traditionally represented by their associated Butcher tableaux, with the notation $c_i=\sum_{j=1}^s a_{ij}$:
\[
\begin{array}{c|c}
    c & A \\
    \hline
     & b
  \end{array}
\]
Similarly to the order theory of Runge-Kutta methods for ODEs, the Taylor expansion of $\psi^{t\xi}(\zeta_0)$ in \eqref{equation:RK_HJ} writes as a Butcher series
\[
\psi^{t\xi}= B^{\xi}\bigl(a\bigr)(0),
\]
with the same coefficient map $a \in \TT^*$ than for standard Runge-Kutta methods. We refer to \cite[Chap.\ts III]{Hairer06gni} for the exact expression of $a$ and an example is
\[
a(\forestEE)=\sum_{i,j,k,l,m=1}^s b_i a_{ij} a_{ik} a_{kl} a_{km}=\sum_{i,k=1}^s b_i c_i a_{ik} c_k^2.
\]
The algebraic reformulation of Section \ref{sec:pre_lie} allows us to take over the classical order theory of Runge-Kutta methods for ODEs \cite{Hairer06gni} and to adapt it in the context of approximations of formal solutions to Hamilton-Jacobi equation. The following theorem provides new Hamiltonian Poisson integrators.
\begin{theorem}
\label{thm:order_cond}
Let us consider a RKHP method \eqref{equation:RK_HJ} with coefficient map $a \in \TT^*$ and set $e \in \TT^*$ as in the equation \eqref{eq:e}. Then, if $a(\tau)=e(\tau)$ for all $\tau\in \TT^N$, the method has order $N$ for solving \eqref{eq:Ham_eq}.
Moreover, under the degeneracy condition \eqref{eq:degen}, if $a(\tau)=e(\tau)$ for all $\tau\in \TT_0^N$, the method has order $N$ for solving \eqref{eq:Ham_eq}.
\end{theorem}

Theorem \ref{thm:order_cond} allows us to recover the standard order conditions of Runge-Kutta methods \cite[Chap.\ts III]{Hairer06gni} in the general case. Moreover, we provide in Table \ref{table:order_conditions} the order conditions in the degenerate case \eqref{eq:degen}. These tree degeneracies are new to the best of our knowledge.
\begin{longtable}{C|C|C}
\text{Order} &\text{Butcher tree } \tau  &\text{Order condition } a(\tau)=e(\tau)\\\hline
1 & \forestFE & \sum_{i=1}^s b_i=1\\
\hline
3 & \forestGE & \sum_{i=1}^s b_i c_i^2 = \frac{1}{3}\\
\hline
4 & \forestHE & \sum_{i,j=1}^s b_i a_{ij} c_j^2 = \frac{1}{12}\\
 & \forestIE & \sum_{i=1}^s b_i c_i^3 = \frac{1}{4}\\
 \hline
5 & \forestJE & \sum_{i,j,k=1}^s b_i a_{ij} a_{jk}c_k^2=\frac{1}{60}\\
 & \forestKE & \sum_{i,j=1}^s b_i a_{ij} c_j^3=\frac{1}{20}\\
 & \forestLE & \sum_{i,j=1}^s b_i c_i a_{ij} c_j^2=\frac{1}{15}\\
 & \forestME & \sum_{i=1}^s b_i c_i^4=\frac{1}{5}\\
\caption{Order conditions of RKHP integrators under the degeneracy condition \eqref{eq:degen}.}
\label{table:order_conditions}
\end{longtable}

\begin{remark}
The idea of approximating bi-realisations at high order has been explored by \cite{cabrera2024}, where a procedure to asymptotically compute the source map $\alpha \colon \GG \to M$ is investigated. This relies on a long history research of deformation theory in order to construct the local symplectic groupoid of any Poisson manifold \cite{Cattaneo2005, cabrera2020}.  We expect a relation between the order of the approximation of the bi-realisation and the one of the dynamics to provide robust numerical methods.
\end{remark}

In the case of a general bi-realisation, a collection of explicit RKHP integrators can be derived straightforwardly from the standard Runge-Kutta integrators from \cite{Hairer06gni}.
In the case of a bi-realisation satisfying \eqref{eq:degen}, we propose a collection of new methods with minimal number of stages for fixed order. The superconvergence property of Theorem \ref{thm:superconvergence} allows us to reach high order $p$ with less than $p$ evaluations.
We emphasize that there is no need to consider discretisations associated to symplectic methods (like the midpoint method) as the geometry has been taken care of by the bi-realisations.

\paragraph*{Euler method:}
The Euler method is a first order RKHP method (resp.\ts second order under \eqref{eq:degen}).
\[
 \begin{array}{rl}
y_n &= \alpha(\Delta t \cdot d_{x_n} H)\\
y_{n+1} &= \beta(\Delta t \cdot d_{x_n} H) 
  \end{array}
\quad\quad
  \begin{array}{c|c}
    0 & 0 \\
    \hline
     & 1
  \end{array}
\]

\paragraph*{Third order method:}
Under \eqref{eq:degen}, third order can be achieved with two stages.
\[
 \begin{array}{rl}
 Z&=\frac{1}{\sqrt{3}}\Delta t \cdot d H \\
y_n &= \alpha(\Delta t \cdot d_{x_n} Z^*\alpha^* H)\\
y_{n+1} &= \beta(\Delta t \cdot d_{x_n} Z^*\alpha^* H) 
  \end{array}
\quad
\quad
\begin{array}{c|cc}
    0 & 0 & 0 \\
    \frac{1}{\sqrt{3}} & \frac{1}{\sqrt{3}} & 0 \\
    \hline
     & 0 & 1
\end{array}
\]

\paragraph*{Fourth order method:}
Under \eqref{eq:degen}, fourth order can be achieved with three stages, that is one stage fewer than for the popular RK4 discretisation.
\[
\begin{array}{rl}
Z_1&=-\frac{\sqrt{3}}{4}\Delta t \cdot d H \\
Z_2&=\frac{3}{4}\Delta t \cdot d Z_1^*\alpha^* H \\
y_n &= \alpha\left(\Delta t ( \frac{11}{27} d_{x_n} H + \frac{16}{27} d_{x_n} Z_2^*\alpha^*H)\right)\\
y_{n+1} &= \beta\left(\Delta t ( \frac{11}{27} d_{x_n} H + \frac{16}{27} d_{x_n} Z_2^*\alpha^*H)\right)
  \end{array}
\quad \quad
  \begin{array}{c|ccc}
    0 & 0 & 0 & 0 \\
    -\frac{\sqrt{3}}{4} & -\frac{\sqrt{3}}{4} & 0 & 0 \\
    \frac{3}{4} & 0 & \frac{3}{4} & 0 \\
    \hline
     & \frac{11}{27} & 0 & \frac{16}{27}
  \end{array}
 \]

\begin{remark}
If implicit implementations are computationally feasible and \eqref{eq:degen} is satisfied, a one-stage third order implicit method is
\[
\begin{array}{rl}
Z&=\frac{1}{\sqrt{3}}\Delta t \cdot d Z^*\alpha^* H\\
y_n &= \alpha(\Delta t \cdot  d_{x_n} Z^*\alpha^* H)\\
y_{n+1} &= \beta(\Delta t \cdot  d_{x_n} Z^*\alpha^* H) 
  \end{array}
  \quad\quad
  \begin{array}{c|c}
    \frac{1}{\sqrt{3}} & \frac{1}{\sqrt{3}} \\
    \hline
     & 1
  \end{array}   
\]
\end{remark}

\begin{remark}
The Butcher-Connes-Kreimer Hopf algebra used in Section \ref{sec:comp} is also relevant to approximations of bisections in $\GG$. Since they admit a group law, a computational consequence is the existence of composition methods.
For instance, let $\xi \in \Omega^1(\GG),$ $\zeta_0 \in \Omega^{1, \GG}(M)$ and let the first order approximation of the Farmer series, analogous to the Euler method, be given by
\[
\psi^{t\xi}(\zeta_0):=\zeta_0+t\zeta_0^* \xi= \zeta_0 + B^{\xi}\bigl(\delta_{\forestNE}\bigr)(\zeta_0).
\]
Then the explicit midpoint method is the composition
\begin{align*}
\hat{\psi}^{t\xi}(\zeta_0)&:= (\psi^{t\xi}\circ \psi^{t\xi/2})(\zeta_0)\\
&=\zeta_0+t \xi\Bigl(\zeta_0+\frac{t}{2} \cdot \zeta_0^*\xi \Bigr)\\
&=\zeta_0+B^{\xi}\Bigl(\delta_{\forestOE}* (\delta_{\forestPE}/2)\Bigr)(\zeta_0)\\
&=\zeta_0 + t \zeta_0^*\xi 
+ \frac{t^2}{2}\zeta_0^*[ \xi, \pi^* \zeta_0^* \xi ]
+ \frac{t^3}{8}\zeta_0^*[[ \xi, \pi^* \zeta_0^* \xi ], \pi^* \zeta_0^* \xi]
+\dots
\end{align*}
We observe in particular that $\hat{\psi}$ provides a second order approximation for a general $\xi \in \Omega^1(\GG)$, similarly to the context of ODEs.
\end{remark}

\section{Conclusion}

The algebraic tools of geometric integration extend for the study of Poisson geometry. In the framework of symplectic groupoids, they bring new insights, from both geometric and computational viewpoints.

Some stability properties of low order Hamiltonian Poisson integrators have been studied in \cite{cosserat2025}. The present work allows one to implement high order Hamiltonian Poisson integrators as well. We also plan to implement implicit methods and benchmark their orders in order to provide numerical illustrations of our algebraic results. Steep dynamical systems of conservative mechanics may benefit from the formalism we introduced; those ideas will deserve a more mechanics oriented article.

On the algebraic side, we expect this work to open many perspectives. A natural extension could adapt the Hopf algebra of substitution for the backward error analysis in this context and to pave the way towards the combinatorial foundations of Hamiltonian Poisson integrators. At last, a more general geometric context could yield post-Lie algebras.
This is matter for future work.

\bigskip

\noindent \textbf{Acknowledgements.}
The authors would like to thank Camilo Angulo, Anton Fehnker, Camille Laurent-Gengoux, Jean-David Jacques, Chenchang Zhu, the AGATA seminar of Montpellier and the Oberseminar on algebraic topology of Göttingen for helpful discussions.
The authors acknowledge the support of the French programs ANR-25-CE40-2862-01 (MaStoC - Manifolds and Stochastic Computations) and ANR-11-LABX-0020-0 (Labex Lebesgue) and the Deutsche Forschungsgemeinschaft RTG 2491.

\bibliographystyle{abbrv}
\bibliography{Ma_Bibliographie, Oscar_Bibliographie}

\end{document}